\documentclass[10pt,a4paper,leqno,twoside]{amsart}
\usepackage{amsmath,amssymb,amsthm,enumerate,hyperref,color,esint}
\numberwithin{equation}{section}
\newtheorem{theorem}{Theorem}[section]

\newtheorem{proposition}[theorem]{Proposition}
\newtheorem{lemma}[theorem]{Lemma}
\newtheorem{corollary}[theorem]{Corollary}
\newtheorem{remark}[theorem]{Remark}

\newcommand{\rad}{{\text{\upshape rad}}}
\newcommand{\loc}{{\text{\upshape loc}}}

\def\e{{\varepsilon}}

\def\L{{\Lambda}}
\def\l{{\lambda}}
\def\a{{\alpha}}

\def\de{\partial}

\newcommand{\R}{\mathbb{R}}

\definecolor{darkgreen}{rgb}{0.0, 0.5, 0.2} 
\definecolor{purple}{rgb}{0.5, 0.0, 0.5}
\newcommand{\AL}{\color{purple}}

\newcommand{\n}{\mathop{N}}

\providecommand{\norm}[1]{\lVert#1\rVert}

\newcommand{\remove}[1]{}

\def\sideremark#1{\ifvmode\leavevmode\fi\vadjust{\vbox to0pt{\vss
 \hbox to 0pt{\hskip\hsize\hskip1em
 \vbox{\hsize2.1cm\tiny\raggedright\pretolerance10000
  \noindent #1\hfill}\hss}\vbox to15pt{\vfil}\vss}}}%

\newtheorem*{theorem*}{Theorem}

\begin{document}

\title[Singular eigenvalue problem] {On a singular eigenvalue problem and its applications in computing the Morse index of solutions to semilinear PDE's}
\author[A.~L.~Amadori, F.~Gladiali]{Anna Lisa Amadori$^\dag$,  Francesca Gladiali$^\ddag$}
\thanks{\small{This work was supported by Gruppo Nazionale per l'Analisi Matematica, la Probabilit\`a e le loro Applicazioni (GNAMPA) of the Istituto Nazionale di Alta Matematica (INdAM). {The second author is supported by Prin-2015KB9WPT}}}
\date{\today}
\address{\small{$\dag$ Dipartimento di Scienze e Tecnologie, Universit\`a di Napoli ``Parthenope", Centro Direzionale di Napoli, Isola C4, 80143 Napoli, Italy. \texttt{annalisa.amadori@uniparthenope.it}}}
\address{\small{$\ddag$ Dipartimento di Chimica e Farmacia, Universit\`a di Sassari, via Piandanna 4, 07100 Sassari, Italy. \texttt{fgladiali@uniss.it}}}

\begin{abstract}
	We investigate nodal radial solutions to semilinear problems of type
	\[
	\left\{\begin{array}{ll}
	-\Delta u = f(|x|,u) \qquad & \text{ in } \Omega, \\
	u= 0 & \text{ on } \partial \Omega,
	\end{array} \right.
	\]
	where $\Omega$ is a bounded radially symmetric domain of $\R^N$ ($N\ge 2$)
	and $f$ is a real function.
	We characterize both the Morse index and the degeneracy in terms of a singular one dimensional eigenvalue problem, which is studied in full detail.  
	The presented approach also describes the symmetries of the eigenfunctions.
This characterization enables to give a lower bound for the Morse index in a forthcoming work. 
\end{abstract}

\maketitle




\section{Introduction}
This paper is aimed to furnish a tool to compute the Morse index and investigate degeneracy of solutions to semilinear elliptic problems 
\begin{equation} \label{general-f}
\left\{\begin{array}{ll}
-\Delta u = f(|x|,u) \qquad & \text{ in } \Omega, \\
u= 0 & \text{ on } \partial \Omega,
\end{array} \right.
\end{equation} 
where $\Omega$ is a bounded radially symmetric domain of $\R^N$, with $N\ge 2$ and $f$ is a real function.
\\
This issue has various interesting applications. Let us recall that the degeneracy points, associated with a change in the Morse index, give existence of other solutions that can be obtained by bifurcation and can give rise to the so called symmetry breaking phenomenon.
In the variational setting, indeed, there is a direct link between the second derivative of the 
functional associated to \eqref{general-f} and the quadratic form related to its linearization, and a change in the Morse index immediately produces a change in the critical groups, giving existence of bifurcating solutions; we refer to \cite{BSW} for the definition of critical groups, and their relation with the Morse index.
But also when the problem does not have a variational structure, as for instance when $f$ is supercritical, a change in the Morse index implies a bifurcation result, via the Leray Schauder degree (see \cite{AM}).  
 An application of this type can be found in  \cite{AG-henon}, dealing with positive solutions of the H\'enon problem.
\\
An estimate of the Morse index of radial nodal solutions has been used in \cite{AP}, for autonomous nonlinearities, to prove that any least energy nodal solution is not radial, thus producing a multiplicity result.
Moreover the knowledge of the Morse index in symmetric spaces turns useful in the case of multiple bifurcation to distinguish solutions as can be seen in \cite{GGPS} and \cite{AG17} dealing respectively with positive and nodal solutions for the Lane Emden problem in an annulus. It also allows to produce nonradial solutions by minimization, as done in \cite{GI} for the Lane-Emden problem in the disk, and extended to the H\'enon problem  in  \cite{AG18-2}, \cite{A}, by taking advantage of the characterization by singular eigenvalues presented here.
Specularly, the nondegeneracy of solutions allows to prove existence results in perturbed domains also in the supercritical setting, as for instance in \cite{C}, \cite{AGG}  and \cite{AG18}.

Our purpose is to characterize Morse index and degeneracy of a solution $u$ by means of a singular eigenvalue problem associated to the linearization of \eqref{general-f}.
The disadvantage of dealing with a singular problem is  rewarded by the fact that, when $u$ is radial, the singular eigenfunctions easily project along the spherical harmonics, thus reducing the issue to a (singular) Sturm-Liouville problem. The presented approach discloses information on the symmetries of the  eigenfunctions and hence on the Morse index and degeneracy  also in spaces of  symmetric functions. In this first paper we extend the main properties of standard eigenvalue problems to the singular ones arising in this framework, and obtain an easier characterization of the Morse index. In a second paper \cite{AG-sing-2} we use this characterization, together with the properties of singular eigenvalues established here,  to give a general lower bound for the Morse index of radial solutions to some problems of H\'enon type (i.e. whose nonlinearity has the form $f(|x|, u)= |x|^{\a} f(u)$).

\

In this paper we are not interested in existence results for \eqref{general-f} nevertheless we remind that radial and radial nodal solutions have been found in \cite{Bartsch-Weth}, \cite{BW93}, \cite{CastroCossioNeuberger},  \cite{BCW} and \cite{Struwe} under some assumptions on $f$. To give an idea we report the ones in \cite{Bartsch-Weth}:
\begin{itemize}
		\item[$(f_1)$] $f\in C^1(\Omega\times \R,\R) \ , \ f(x,0)=0$ for all $x\in \Omega$;
		\item[$(f_2)$] there exists $  p \in \big(2, \frac {2N}{N-2}\big]$ (or  $p\in (2,\infty)$ in case $N=2$)  such that $|f_u(x,t)|\leq C(1+|t|^{p-2})$  for all $x\in \Omega, \ t\in \R$;
		\item[$(f_3)$] $f_u(x,t)>f(x,t)/t \ \text{ for all }x\in \Omega, \ t\neq 0$;
		\item[$(f_4)$] there exists $R>0 \text{ and }\theta>2 \text{ such that }0<\theta \int_0^t f(x,s)\, ds \leq tf(x,t) \text{ for all }x\in \Omega, \ |t|\geq R$.
	\end{itemize}
	Here we want to include  also  some problems of H\'enon type,  for which the existence of radial solutions is established under less restrictive growth conditions than $(f_2)$. 
Thus we shall take as underlying assumption that 
\begin{enumerate}[{A.}1]
\item $f(r,u):[0,\infty)\times \R \to \R$  is continuous w.r.t.~$r$ and for all $r\ge 0$ the function $t\mapsto f(r,t)$ belongs to $W^{1,1}_{\loc}(\R)$,
\end{enumerate}
and consider weak solutions $u\in H^1_0(\Omega)$ to \eqref{general-f} such that 
\begin{enumerate}[{A.}2]
\item  $f_u(|x|,u(x))\in L^{\infty}(\Omega)$.
\end{enumerate}
 where  by  weak solution we mean a function $u\in H^1_0(\Omega)$ s.t.
	\[\int_\Omega \left(\nabla u\nabla \varphi  -f(|x|,u)\varphi \right) dx=0\]
	for any test function $\varphi\in H^1_0(\Omega)$.
It is easily seen that hypothesis  {A.2} is satisfied by classical solutions, under assumption $(f_1)$, as well as by weak solutions, under assumptions $(f_1)$--$(f_4)$, thanks to elliptic regularity theory.
Assumptions {A.1} and {A.2} are needed respectively to give a sense and recover compactness of the linearized operator  associated to \eqref{general-f}
\begin{align*}
L_u(\psi)&:=-\Delta \psi-f_u(|x|, u)\psi .
\end{align*}
for $\psi\in H^1_0(\Omega)$. A
 weak solution $u$ to \eqref{general-f} is said degenerate if the linearized equation $L_u(w)$ admits a nontrivial weak solution $w\in H^1_0(\Omega)$, equivalently if the linearized operator $L_u$ admits $0$ as an eigenvalue in $H^1_0(\Omega)$. The Morse index of $u$ 
is instead the maximal dimension of a subspace of $H^1_0(\Omega)$ in which the quadratic form
\[
{\mathcal Q}_u (w):=\int_{\Omega}\left(|\nabla w|^2 - f_u(|x|,u(x))\,w^2\right) dx.
\]
is negative defined, or equivalently, since  $L_u$ is a linear compact operator, is the number of the negative eigenvalues of  $L_u$ in $H^1_0(\Omega)$, counted with multiplicity.
When $u$ is radial, the linearized operator $L_u$ and the quadratic form $\mathcal Q_u$ can be regarded in some subspace $H^1_{0,\mathcal G}(\Omega)$ of $H^1_0(\Omega)$ given by functions which are invariant by the action of a subgroup $\mathcal G$ of the orthogonal group $O(N)$. Then we can say that $u$ is $\mathcal{G}$-degenerate if $L_u$ admits 0 as an eigenvalue in $H^1_{0,\mathcal G}(\Omega)$ and the $\mathcal{G}$-Morse index of $u$ is the number, counted with multiplicity, of the negative eigenvalues of $L$ in $H^1_{0,\mathcal G}(\Omega)$. 
In particular when $\mathcal G =O(N)$ we say that $u$ is radially degenerate if $L_u$ admits 0 as an eigenvalue in $H^1_{0, \rad}(\Omega)$ (the
subspace of $H^1_0 (\Omega)$ given by radial functions) and the radial Morse index of $u$ is the
number of the negative eigenvalues of $L_u$ in $H^1_{0, \rad}(\Omega)$.
\\
Then both of our aims, Morse index and degeneracy,
are described in terms of the eigenvalues of a linear compact operator.
In this perspective in Section \ref{se-prel} we recall the main properties of standard eigenvalues and eigenfunctions associated to a compact operator of type
\begin{align*}
L_a w :=-\Delta w-a(x)w ,\quad
\intertext{and the related quadratic form}
\label{forma-quadratica-a}
Q_a(w) :=\int_{\Omega} \left(|\nabla w|^2-a(x)w^2\right) dx, \quad
\end{align*}
where $a(x)$ is a generic function, $a\in L^{\infty}(\Omega)$,
and we address a very classical linear eigenvalues problem: 
\begin{equation}\label{Rayleigh} \begin{split}
\L_1 & :=\min\left\{ \frac{Q_a(w)}{\int_{\Omega} w^2(x)\, dx}  : \, w\in  H^1_0({\Omega} )\setminus\{0\}\right\} ,  \\
\L_i & :=\min\left\{ \frac{Q_a(w)}{\int_{\Omega} w^2(x)\, dx}  : \, w\in  H^1_0({\Omega} )\setminus\{0\}, \,  w\perp \psi_1,\dots,\psi_{i-1} \right\} \quad   \text{ as } i\ge 2,
\end{split}\end{equation}
where the orthogonality condition $w\perp \psi_j$ stands for the orthogonality in $L^2({\Omega} )$. 
\\
When the function $a$ is radial, we also look at its radial version
\begin{equation}\label{Rayleigh-rad} \begin{split}
\L_1^{\rad}&  :=\min\left\{ \frac{\mathcal Q_a(w)}{\int_{\Omega}  w^2(x)\, dx}  : \, w\in  H^1_{0,\rad}({\Omega} )\setminus\{0\} \right\} ,  \\
\L_i^{\rad} & :=\min\left\{ \frac{\mathcal Q_a(w)}{\int_{\Omega}  w^2(x)\, dx}  : \, w\in  H^1_{0,\rad}({\Omega} )\setminus\{0\}, \,  w\perp \psi_1^{\rad},\dots,\psi_{i-1}^{\rad} \right\} \quad   \text{ as } i\ge 2 ,
\end{split}\end{equation}
which is related to the Sturm-Liouville problem 
\begin{equation}\label{radial-eigenvalue-problem}
\left\{\begin{array}{ll}
-\left(r^{N-1}\left(\psi_i^{\rad}\right)'\right)' -r^{N-1} a(r) \psi_i^{\rad} =r^{N-1} \L_i^{\rad}\psi_i^{\rad} & \text{ for } r\in(0,1)\\
{ (\psi_i^{\rad})'(0)=0 , \; \psi_i^{\rad}(1)=0 , } 
\end{array} \right.
\end{equation}
if $\Omega$ is the unit ball centered at the origin, or
\begin{equation}\label{radial-eigenvalue-problem-annulus}
\left\{\begin{array}{ll}
-\left(r^{N-1}\left(\psi_i^{\rad}\right)'\right)' -r^{N-1} a(r) \psi_i^{\rad} =r^{N-1} \L_i^{\rad}\psi_i^{\rad} & \text{ for } r\in(a,b)\\
\psi_i^{\rad}(a)= \psi_i^{\rad}(b) = 0 , 
\end{array} \right.
\end{equation}
if $\Omega$ is an annulus centered at the origin. 

In  Section \ref{se-prel} we recall the main properties of standard eigenvalues and eigenfunctions associated to a compact operator, \eqref{Rayleigh} and \eqref{Rayleigh-rad}. 
The points concerning radial eigenfunctions are addressed in the framework of weak solutions to the Sturm-Liouville problem \eqref{radial-eigenvalue-problem} in view of the subsequent extension to the singular setting. 
Next in Section \ref{se:singular} we introduce a singular eigenvalue problem associated to $L_a$, which requires an ad hoc  choice of the functional setting.
Letting $\mathcal{L}:=\{v:\Omega\to \R\ : \ v/|x| \in L^2(\Omega) \}$,  we define 
$\mathcal {H}_{0}:=H^1_0(\Omega)\cap \mathcal{L}$ and $\mathcal {H}_{0,\rad}:=H^1_{0, \rad}(\Omega)\cap \mathcal{L}$ and 
\begin{equation}\label{i+1-singular}\begin{split}
\widehat \L_{1}& := \inf \left\{ \frac{\mathcal Q_a(w)}{\int_{\Omega} |x|^{-2}w^2(x)\, dx}  : \, w\in  {\mathcal H}_0\setminus\{0\}\right\}, \\
\widehat \L_{i}& := \inf \left\{ \frac{\mathcal Q_a(w)}{\int_{\Omega} |x|^{-2}w^2(x)\, dx}  : \, w\in  {\mathcal H}_0\setminus\{0\}, \,  w\underline{\perp} \widehat \psi_1,\dots,\widehat \psi_{i-1} \right\} \quad \text{ as } i\ge 2,
\end{split}\end{equation}
where the orthogonality stands for the orthogonality in $\mathcal L$ (see Section \ref{se:singular}),
and its radial version
\begin{equation}\label{radial-singular}\begin{split}
\widehat \L_{1}^{\rad}& :=\inf \left\{ \frac{\mathcal Q_a(w)}{\int_{\Omega} |x|^{-2}w^2(x)\, dx}  : \, w\in  {\mathcal H}_{0,\rad}\setminus\{0\} \right\}, \\
\widehat \L_{i}^{\rad}& :=\inf \left\{ \frac{\mathcal Q_a(w)}{\int_{\Omega} |x|^{-2}w^2(x)\, dx}  : \, w\in  {\mathcal H}_{0,\rad}\setminus\{0\}, \,  w\underline{\perp} \widehat \psi_1^{\rad},\dots,\widehat \psi_{i-1}^{\rad} \right\} \quad  \text{ as } i\ge 2 .
\end{split}\end{equation}
This last minimization problem is related to a Sturm-Liouville problem which is singular when $\Omega=B$ is the unit ball centered at the origin
	\begin{equation}\label{radial-singular-problem}
	\left\{\begin{array}{ll}
	-\left(r^{N-1}(\psi_i^{\rad})'\right)'- r^{N-1}a(r) \psi_i^{\rad}
	= r^{N-3}\widehat \L_i^{\rad} \psi_i^{\rad} & \text{ for } r\in (0,1)\\
 \psi _i^{\rad}\in \mathcal{H}_{0,\rad} .
	\end{array} \right.
	\end{equation} 
{When $\Omega$ is an annulus the problems  \eqref{i+1-singular}, \eqref{radial-singular} and \eqref{radial-singular-problem} are regular and equivalent to the eigenvalue problems studied in Section \ref{se-prel}. Indeed}  it is easy to see that there is a strict correspondence between nonpositive eigenvalues of  \eqref{Rayleigh} and of \eqref{i+1-singular} (and similarly for their symmetric and radial versions). See, for instance, \cite{AGG}, where this fact has been exploited to investigate the Lane-Emden problem in annular domains.
For these reasons in all the following we shall only consider $\Omega$ to be the unitary ball which is the main interesting case.

The singular problems \eqref{i+1-singular}, \eqref{radial-singular}  have been studied before in \cite{GGN2} for $N=2$ and \cite{DGG} when $N\ge 3$  where the eigenvalues $\widehat \L_i$ and $\widehat \L_i^\rad$ have been characterized, when they are negative, despite a lack of compactness. Here
we present a unified and improved study, providing among other things a sharp condition under which the minima in \eqref{i+1-singular}, \eqref{radial-singular} are attained: indeed compactness can be restored  as far as $\widehat\Lambda_i , \widehat{\Lambda}_i^{\rad}<\left(\frac{N-2}{2}\right)^2$. 
Furthermore we give a detailed description of the behaviour of the eigenfunctions close to the singular point $x=0$. 
Next the singular eigenvalues and eigenfunctions enjoy the basic properties of the regular ones: in particular the radial singular eigenvalues are simple, the eigenfunction related to the $i^{th}$ eigenvalues has exactly $i$ nodal domains, and Picone  identity holds (see Properties 1--5 in Subsection \ref{sse:general}).
In view of the applications to H\'enon type problems  we read \eqref{radial-eigenvalue-problem}, \eqref{radial-singular-problem} in  a broader sense, that does not require $N$ to be an integer, precisely we deal with the Sturm-Liouville problems 
	\begin{equation}\label{radial-eigenvalue-problem-H}
\left\{\begin{array}{ll}
- \left(t^{M-1} \phi'\right)'- t^{M-1} a(t)\, \phi = t^{M-1} {\nu} \, \phi & \text{ for } t\in(0,1)\\
\phi'(0)=0 , \ \phi(1)=0  , &
\end{array} \right.
\end{equation} 
	\begin{equation}\label{radial-singular-problem-H}
	\left\{\begin{array}{ll}
	- \left(t^{M-1} {\phi}'\right)'- t^{M-1} a(t) \,  \phi = t^{M-3} \widehat{\nu}  \,  \phi & \text{ for } t\in(0,1)\\
	\phi\in  \mathcal H_{0,M} , &
	\end{array} \right.
	\end{equation} 
	for $M\in\R$, $M\ge 2$.
	Here $\mathcal H_{0,M}$ is the right space where to work with solutions to \eqref{radial-singular-problem-H}
	and it is made of measurable functions $v:(0,1)\to \R$ such that $\int_0^1r^{M-3}v^2dr <\infty$, such that $v$ has a first order weak derivative $v'$ that satisfies $\int_0^1 r^{M-1} (v')^2dr<\infty$. Definition and properties of $\mathcal H_{0,M}$ are given in Subsection \ref{sse:general} and in the Appendix together with a variational definition of the eigenvalues $\widehat\nu_i$. The analogous spaces where to consider solutions to \eqref{radial-eigenvalue-problem-H} are given in Subsection \ref{sse:prel-radial}. 

Afterwards in Section \ref{sse:2-bis} we turn back to the semilinear equation \eqref{general-f} and see that 
\begin{proposition}[Alternative definition of Morse index]\label{prop-1-2bis}       
	Assume {A.1}	and take $u$ a weak  solution to \eqref{general-f} satisfying {A.2}.          The Morse index of $u$ is given by the number, counted with multiplicity of negative singular eigenvalues $\widehat \L$ defined in \eqref{i+1-singular} for $a(x) = f_u (|x|, u(x))$. \par		In the same way, if ${\mathcal G}$ is any subgroup of the orthogonal group $O(N)$ and $u$ is ${\mathcal G}$-invariant, its ${\mathcal G}$-Morse index is given by the number of negative singular radial eigenvalues  
	$\widehat \L$ which have  ${\mathcal G}$-invariant eigenfunctions. \par In particular if $u$ is radial, then its radial Morse index is given by the number of negative singular radial eigenvalues $\widehat \L^{\rad}$ defined in \eqref{radial-singular} for $a(x) = f_u (|x|, u(x))$. 
\end{proposition}
This, together with a decomposition result  in Section \ref{sse:2-bis}, brings to explicit formulas to compute the Morse index and evaluate the degeneracy of radial solutions in terms of the singular Sturm-Liouville problem.
To give a precise statement we let $\l_j:=j(N+j-2)$ be the eigenvalues of the Laplace-Beltrami operator in the sphere of dimension $N-1$,  $N_j=\frac{(N+2j-2)(N+j-3)!}{(N-2)!j!}$ their multiplicity, and  $Y_j$ their eigenfunctions. Moreover we write $\L^{\rad}_i$ and  $\widehat\L^{\rad}_i$ meaning the radial eigenvalues in \eqref{Rayleigh-rad} and \eqref{radial-singular} with $a(r)= f_u(r, u(r))$.

\begin{theorem}[Morse index formula]\label{general-morse-formula} 
	Assume {A.1} and take $u$ a radial weak  solution to \eqref{general-f} satisfying {A.2}, and let $m_{\rad}=m_{\rad}(u)$ its radial Morse index. Then the Morse index of $u$ is given by
	\begin{equation}\label{tag-2}\begin{split}
	m(u)= & \sum\limits_{i=1}^{m_{\rad}}
	 \sum\limits_{j=0}^{\lceil J_i -1 \rceil} N_j  \qquad \qquad  \mbox{where} \\
	J_i= & \sqrt{\left(\frac{\n-2}{2}\right)^2-\widehat\L^{\rad}_i}-\frac{\n-2}{2}  
	\end{split}\end{equation}
	and $\lceil t \rceil = \min\{ k\in {\mathbb Z} \, : \, k\ge t\}$ stands for the ceiling function.
		Besides the negative singular eigenvalues are $\widehat\L = \widehat\L^{\rad}_i +\lambda_j$ 
		and the related eigenfunctions are, in spherical coordinates, 
		\begin{equation}\label{decomp-autofunz-f(u)} \psi(x)= \psi_i^{\rad}(r)Y_j(\theta), \end{equation} 
		where $ \psi_i^{\rad}$ is an eigenfunction related to $\widehat\L^{\rad}_i$.	
\end{theorem}

Degeneracy, i.e.  eigenvalue 0, deserves a particular care. The reason is that when $N\ge 3$  the weighted space ${\mathcal H}_{0}$  coincides with $H^1_0(B)$  thanks to Hardy inequality, while in dimension $N=2$ the inclusion is strict.

\begin{proposition}[Characterization of degeneracy]\label{non-degeneracy}
	Assume { A.1} and take $u$ a radial weak  solution to \eqref{general-f} satisfying { A.2}.
	When $N\ge 3$  $u$ is radially degenerate if and only if $\widehat \L_k^{\rad} =\L_k^{\rad}=   0$ for some $k\ge 1$,
	and degenerate if and only if, in addition,
	\begin{equation}\label{non-radial-degeneracy}
	\widehat \L_k^{\rad} =  -j (N-2+j)\qquad \mbox{for some $k, j \ge 1$.}
	\end{equation}
	Otherwise if $N$=2 $u$ is radially degenerate if and only if $\L_k^{\rad} =  0$ for some $k\ge 1$, and degenerate if and only if, in addition, \eqref{non-radial-degeneracy} holds.
	\\
	Besides in any dimension $N\ge 2$,  any nonradial function in the kernel of $L_u$ can be written as in \eqref{decomp-autofunz-f(u)}, where $ \psi_k^{\rad}$ is an eigenfunction related to an eigenvalue $\widehat \L_k^{\rad}$ satisfying \eqref{non-radial-degeneracy}.	
\end{proposition}

Formulas like \eqref{tag-2} and \eqref{non-radial-degeneracy} have already been exploited to compute explicitly the Morse index in some particular cases, as for instance in \cite{GGPS} and \cite{AGG}  \cite{DGG}, \cite{GGN}, and more recently in \cite{KW}.  
Further in many situations where a limit problem is identified, they allow to compute the Morse index of the solutions for that values of the parameter which are near to the limit. This is the case of  \cite{DIPa} and \cite{DIPb}, both dealing with the Lane-Emden problem. 
The difficulty coming from singularity at the origin has been bypassed  there by approximating the ball by annuli with a small hole. This trick forces to make very accurate estimates to control the approximation parameter, especially when dealing with asymptotic estimates and, in our opinion, it is not needed if one uses the singular eigenvalues and \eqref{tag-2}.

The general characterizations provided by Theorem \ref{general-morse-formula} and Proposition \ref{non-degeneracy}, together with the properties of the singular eigenvalues collected in Subsection \ref{sse:general}, have various interesting consequences. The framework of applications we have in mind is the one of H\'enon type problems
\begin{equation} \label{general-f-H}
\left\{\begin{array}{ll}
-\Delta u = |x|^{\alpha}f(u) \qquad & \text{ in } \Omega, \\
u= 0 & \text{ on } \partial \Omega,
\end{array} \right.
\end{equation}
where  $\a\geq 0$ is a  real parameter.
In this case assumptions A.1 and A.2 simplify into
\begin{enumerate}[{H.}1]
	\item $f\in W^{1,1}_{\loc}(\R)$,
\end{enumerate}
and $u\in H^1_0(\Omega)$ is a weak solution to \eqref{general-f-H} such that 
\begin{enumerate}[{H.}2]	\item  $a(x):=f'(u(x))\in L^{\infty}(\Omega)$.
\end{enumerate}
Notice that {H.1} and {H.2} are fulfilled under the general existence assumptions in \cite{Ni}, because in that framework weak solutions are indeed classical solutions,  see \cite[Proposition 4.1]{AG-sing-2}. 
The strength of the present approach shows up because the arguments valid to deal with the autonomous problem  (i.e. the case $\a=0$)
can be generalized to H\'enon type problems \eqref{general-f-H} in any dimension and with relatively small effort. 
Actually the transformation
	\begin{equation}\label{transformation-henon-intro}
	t=r^{\frac{2+\a}{2}} , \qquad w(t)=u(r) ,
	\end{equation} 
	is used to pass from a radial solution $u$ of 
	\eqref{general-f-H} in the ball  to a  solution $w$ of  
	\begin{equation}\label{lane-emden-radial}
	\begin{cases}
	- \left(t^{M-1} w^{\prime}\right)^{\prime}= t^{M-1} \left(\frac{2}{2+\a}\right)^2 f(w)  , \qquad  & 0< t< 1, \\
	w'(0)=0, \; w(1) =0 , &
	\end{cases}\end{equation}
	with $M =\frac{2(N+\alpha)}{2+\alpha}$.
	\\
	When $M$ is an integer, \eqref{lane-emden-radial} is nothing else that the radial version of a problem of type \eqref{general-f-H} of autonomous type, precisely with $\a=0$ and $f(u)$ replaced by  $\left(\frac{2}{2+\a}\right)^2 f(u)$. But the study performed in Sections \ref{se-prel}, \ref{se:singular} allows to manage this situation for every value of $M\ge 2$, and so to characterize both the Morse index and the degeneracy of a radial solution in terms of some generalized radial singular eigenvalues.
In the following statements we will denote by ${\nu}_i$ (resp., $\widehat{\nu}_i$) the standard (resp., singular) eigenvalues defined as in \eqref{radial-eigenvalue-problem-H} (resp., \eqref{radial-singular-problem-H}) with $M=\frac{2(N+\alpha)}{2+\alpha}$ and $a(t)=\left(\frac{2}{2+\a}\right)^2 f'(w(t))$, being $w(t)$ the function defined in \eqref{transformation-henon-intro}.

\begin{proposition}\label{general-morse-formula-H}
Assume that $\a\geq 0$ and $f$ satisfies $H.1$ and take $u$ a radial weak solution to \eqref{general-f-H} satisfying $H.2$. Then its radial Morse index $m_{\rad}$ is the number of negative eigenvalues of \eqref{radial-singular-problem-H}, and its 
	Morse index is given by
	\begin{align}\label{tag-2-H}\begin{split}
	m(u) = & \sum\limits_{i=1}^{m_{\rad}}\sum\limits_{j=0}^{\lceil J_i -1\rceil } N_j, \quad \qquad \mbox{where} \\
	J_i= & \frac{2+\a}{2} \left(\sqrt{\left(\frac{M-2}{2}\right)^2- \widehat{\nu}_i}-\frac{M-2}{2}\right).
	\end{split}\end{align}
	Furthermore the negative singular eigenvalues are $\widehat\L = \left(\frac{2+\a}{2}\right)^2\widehat\nu_i +\lambda_j$ 
	and the related eigenfunctions are, in spherical coordinates,
	\begin{equation}\label{decomp-autofunz-f(u)-H}
	\psi(x)= \phi_i\big( r^{\frac {2+\a}2}\big) Y_j(\theta), 
	\end{equation} 
	where $ \phi_i$ is an eigenfunction for \eqref{radial-singular-problem-H} related to $\widehat\nu_i$.	
\end{proposition}

\begin{proposition}\label{non-degeneracy-H}
Assume that $\a\geq 0$ and $f$ satisfies $H.1$ and take $u$ a radial weak solution to \eqref{general-f-H} satisfying $H.2$.	When $N\ge 3$ then $u$ is radially degenerate if and only if $\widehat{\nu}_k = {\nu}_k =  0$  for some $k\ge 1$, 
	and degenerate if and only if, in addition,
	\begin{equation}\label{non-radial-degeneracy-H}
	\widehat{\nu}_k =  - \left(\frac{2}{2+\a}\right)^2 j (N-2+j) \qquad \mbox{ for some $k, j\ge 1$.}
	\end{equation}
	Otherwise if $N=2$ then $u$ is radially degenerate if and only if $\nu_k=0$ for some $k\ge 1$, 
	and degenerate if and only if, in addition, \eqref{non-radial-degeneracy-H} holds. \\
	Besides in any dimension $N\ge 2$,  any nonradial function in the kernel of $L_u$ has the form \eqref{decomp-autofunz-f(u)-H}.
\end{proposition}

Afterwards in \cite{AG-sing-2}, by estimating the singular radial eigenvalues, we provide a lower bound for the Morse index of radial solutions to \eqref{general-f-H} which can have an arbitrary number of nodal zones, namely the connected components of $\{x\in \Omega: u(x)\neq 0\}$.

\begin{theorem}\label{morse-estimate-H}
 Assume that $\a\geq 0$ and $f$ satisfies $H.1$, and take $u$ a weak radial solution to \eqref{general-f-H} with $m$ nodal zones, satisfying $H.2$.	Then 
	\begin{align*}	
	& m_{\rad}(u)  \ge m -1 , \\
	& m(u)   \ge   m_{\rad}(u) + (m-1)\sum\limits_{j=1}^{[\frac{2+\alpha}{2}]}N_{j} \ge (m-1)\sum\limits_{j=0}^{[\frac{2+\alpha}{2}]}N_{j} 
	\intertext{If in addition $f'(t)> f(t)/t$ for all $t\neq 0$, then }
	&  m_{\rad}(u) \ge  m , \\
	& m(u)  \ge  m_{\rad}+	(m-1)\sum\limits_{j=1}^{[\frac{2+\alpha}{2}]}N_{j} \ge m+(m-1)\sum\limits_{j=1}^{[\frac{2+\alpha}{2}]}N_{j} 
\end{align*}
\end{theorem}
Here $[ t ] = \max\{ k\in {\mathbb Z} \, : \, k\le t\}$ stands for the integer part.
\\
A far more accurate description is feasible in the particular case of the power nonlinearity, i.e. for the actual H\'enon problem 
\begin{equation} \label{H}
\left\{\begin{array}{ll}
-\Delta u = |x|^{\alpha} |u|^{p-1} u \qquad & \text{ in } \Omega, \\
u= 0 & \text{ on } \partial \Omega,
\end{array} \right.
\end{equation}
whose existence range  in the radial setting is, as known,  $1<p< \frac{N+2+2\a}{N-2}$ in dimension $N\ge 3$ and $1<p$ in dimension $N=2$, see \cite{Ni}.
First, one can see that the radial Morse index of any radial solution is equal to the number of its nodal zones
 and $u$ is radially nondegenerate, see \cite[Theorem 1.3]{AG-sing-2}.
Moreover Proposition \ref{general-morse-formula-H} allows to compute the exact Morse index when the parameter $p$ is at the ends of the existence range by computing the limit of the singular eigenvalues $\widehat\nu_i$.
		More precisely we  shall see that
\begin{theorem}\label{asympt_N3} Let $u$ be a radial solution to \eqref{H} with $m$ nodal zones in $\Omega=B_1(0)\subset \R^N$ with $N\geq 3$. Then there exists $p^\star\in(1, \frac{N+2+2\a}{N-2})$ such that for any $p\in [p^\star,\frac{N+2+2\a}{N-2})$ we have 
	\[\begin{array}{rlr}    
	m(u)& = m \sum\limits_{j=0}^{\left[\frac{2+\alpha}{2}\right]} N_j  
	\quad & \mbox{as $\alpha>0$ is not an even integer, or} \\
	 m(u) & =  m\sum\limits_{j=0}^{\left[\frac{\alpha}{2}\right]} N_j + (m-1) N_{\left[\frac{2+\alpha}{2}\right]} \quad & 	\mbox{if $\alpha=0$ or it is an even number.}	
	\end{array}\]	 
\end{theorem}
Theorem \ref{asympt_N3}	is a consequence of the fact that $\widehat \nu_i\to -(M-1)$ as $p\to \frac{N+2+2\a}{N-2}$ for $i=1,\dots,m$, together with some general estimates obtained in \cite{AG-sing-2}. 
	In dimension $2$ instead the situation is different. Actually for a solution with two nodal zones we have $\widehat \nu_2\to -1$ but $\widehat \nu_1 \to - \kappa^2$, where  $\kappa \approx 5.1869$ is a fixed number that has been characterized in \cite{GGP} when dealing with the Lane-Emden problem. Consequently the asymptotic behaviour of the Morse index is described by next Theorem
\begin{theorem}\label{asympt_N2} Let $u$ be a radial solution to \eqref{H} with $2$ nodal zones in $\Omega=B_1(0)\subset \R^2$.
	For all $\alpha\geq 0$ except the sequence $\a_n= 2(\frac{n}{\kappa}-1)$  there exists $p^\star>1$ such that for any $p>p^\star$ we have
	\[ m(u)  =2\left\lceil\frac {2+\a}2\kappa\right\rceil +  2\left\lceil\frac{\alpha}{2}\right\rceil \]
	where $\lceil t\rceil = \min\{ n \in {\mathbb Z} \, : \, n\ge t\}$ stands for the ceiling function.
\end{theorem}

 In the case $\a=0$, Theorems \ref{asympt_N3} and \ref{asympt_N2} give back the asymptotic Morse index of the Lane-Emden problem computed in  \cite{DIPa} and \cite{DIPb}, respectively. For $\a>0$ the proofs of these two theorems require careful estimates on the asymptotic behaviour of the solution $u$ as well as of the asymptotic of the eigenfunctions associated with $\widehat \nu_i$ that we defer to the papers \cite{AG18} and \cite{AG18-2}.

\section*{List of notations}
	In the following $\Omega$ denotes a generic open subset of $\R^{N}$, while $B=\{x\in\R^N \, : \, |x|<1\}$ is the unit ball.
	We shall make use of the following functional spaces:
	\begin{align*}
  L^2(\Omega) &:= \{ v:\Omega\to \R \, : \, v \mbox{ measurable with } \int_{\Omega} v^2 dx <+\infty\} , 
  \\
  H^1(\Omega)& := \{ v\in L^2(\Omega) \, : \, v \mbox{ has first order weak derivatives } \partial_{i}v \text{ in } L^2(\Omega)  \} , 
 \\
    H^1_0(B)& := \{ v\in H^1(B) \, : \, v(x) =0 \mbox { if } |x|=1 \} , 
  \\
   H^1_\rad(B)&:= \{v\in H^1(B)\, : \, v \text{ radial }\},
  \qquad 
  H^1_{0,\rad}(B) := H^1_0(B)\cap  H^1_\rad(B) , 
 \\
  L^2_M&: = \{v:(0,1)\to\R\, : \, v \text{ measurable and s.t. } \int_0^1 r^{M-1} v^2 dr < +\infty\} ,
  \\
  H^1_M& : = \{v\in L^2_M \, : \, \text{ $v$ has a first order weak derivative $v'$ in }L^2_M \},
  \\
  H^1_{0,M} & : = \left\{ v\in H^1_M \, : \, v(1)=0\right\}  ,
  \end{align*} \begin{align*}
{\mathcal L} & := \{v:  B\to \R\, : \,  v \text{ measurable and s.t } \int_B |x|^{-2}v^2\, dx<\infty\},
\\
\mathcal{H} &:=H^1(B)\cap  {\mathcal L}, \quad \mathcal{H}_0 :=H_0^1(B)\cap \mathcal L,
\quad
\mathcal{H}_{0,\rad} :=  \mathcal{H}\cap H^1_{0,\rad}(B) , 
\\
{\mathcal L}_M &  :=\{ v: (0,1)\to \R \, : \, v  \text{ measurable and s.t } \int_0^1 r^{M-3}v^2\, dr<\infty\} ,
\\
\mathcal{H}_M & :=H^1_M\cap {\mathcal L}_M , \quad 
 \mathcal{H}_{0,M} :=H^1_{0,M}\cap {\mathcal L}_M .
\end{align*}

\

\section{Preliminaries on standard eigenvalues and eigenfunctions}\label{se-prel}

In this section we review the by now standard theory of eigenvalues and eigenfunctions for a compact operator, recalling the proofs of the main points that we are going to extend to  the singular setting.  {To simplify the notations we consider hereafter only the case when $\Omega$ is the unit ball.}

In the remaining of the paper we denote by $a$ a function in $L^{\infty}(B)$, and for any $w\in H^1_0(B)$ we let $L_a(w): H^1_0(B)\to \R$ be the linear operator defined by
\begin{equation}\label{L-a}
L_a(w)\varphi \ := \int_B \nabla w \nabla \varphi-a(x) w \varphi \ dx ,
\end{equation}
and by $\mathcal Q_a: H^1_0(B)\to \R$ the quadratic form associated to $L_a$, namely
\begin{equation}\label{forma-quadratica-a}
 {\mathcal Q}_a(w)=\int_B|\nabla w|^2-a(x)w^2\, dx
\end{equation}
When $a(x)\leq 0$, for any $g\in L^2(B)$ we let $T_a(g) : L^2(B) \to L^2(B)$ to be defined by $T_a(g)=w\in H^1_0(B)$ if $w$ is the unique weak solution to $L_a w=g$ in $B$ with Dirichlet boundary conditions. 
 $T_a$ is a linear selfadjoint positive and compact operator, due to the compact embedding of $H^1_0(B)$ into $L^2(B)$, 
 hence by the classical spectral theory it has a nonicreasing sequence of eigenvalues $\mu_j>0$ such that  $\mu_j\to 0$ as $j\to \infty$, and the corresponding eigenfunctions $\psi_j\in H^1_0(B)$ form an Hilbert basis for $L^2(B)$ (see, for instance, \cite{CH}). Then the values $\L_j:=1/\mu_j$ are called eigenvalues for the linear operator $L_a$ and form a sequence  
$0<\L_1<\L_2\leq \dots$ such that $\L_j\to \infty$ as $j\to \infty$. The eigenfunction
$\psi_i\in H^1_0(B) $ corresponding to $\L_i$, called $i$-th eigenfunction of $L_a$,  is a weak solution to 
\begin{equation}\label{eigenvalue-problem}
\left\{\begin{array}{ll}
-\Delta  \psi_i-a(x)\psi_i =\L_i\psi_i & \text{ in } B\\
\psi_i= 0 & \text{ on } \partial B.
\end{array} \right.
\end{equation}
When $a(x)$ is positive somewhere in $B$ we have that $a(x)-\Gamma\leq 0$ for some $\Gamma>0$ and we can repeat the previous argument for the positive operator $T_{a-\Gamma}$. This produces a sequence of values $-\Gamma<\L_1<\L_2\leq \dots$ such that $\L_j\to \infty$ as $j\to \infty$, called eigenvalues for $L_a$ and the corresponding eigenfunctions 
$\psi_i\in H^1_0(B) $, which form an Hilbert basis for $L^2(B)$, weakly  solve \eqref{eigenvalue-problem}. Summarizing for every $a(x)\in L^{\infty}(B)$ the linear operator $L_a$ admits a sequence of eigenvalues $\L_i$ and an Hilbert basis of eigenfunctions $\psi_i$. 

\begin{remark}\label{regolarita}
By elliptic regularity (see, for instance, \cite[Theorem 8.34]{GT}), the eigenfunctions $\psi$  defined according to \eqref{eigenvalue-problem} belong to $C^{1,\beta}(\bar B)$ for some $0<\beta<1$. Indeed since $a\in L^{\infty}(B)$ and $\psi\in H^1_0(B)$, then $a\psi+\L\psi\in L^q(B)$ for any $q\in (1,\infty)$ when $N=2$ and for any $q\in (1, \frac{2N}{N-2}]$ when $N\geq 3$. So elliptic $L^q$ estimates imply that $\psi\in W^{2,q}(B)$ with $q$ as before, and bootstrapping that $\psi\in W^{2,q}(B)$ for any $q$ in any dimension $N$. The $C^{1,\beta}(\bar B)$ regularity then follows by  the Morrey's Theorem. 
If in addition $a\in C^{0,\a}(B)$ for some $\a>0$,  Schauder estimates yield that $\psi_i\in C^{2,\a}(B)$ so that it is also a classical solution to \eqref{eigenvalue-problem}, corresponding to $\L_i$.
\end{remark}

The classical eigenvalues $\L_i$ can be defined using their min-max characterization, namely 
 \begin{equation}\label{Rayleigh1}
 \L_1 :=\min_{\substack{w\in  H^1_0(B)\\ w\neq 0}} \frac{Q_a(w)}{\int_B w^2(x)\, dx}  
  \end{equation}
and it is attained at a function $\psi_1\in H^1_0(B) $ that is a weak solution  to \eqref{eigenvalue-problem} corresponding to $\L_1$.
  Next iteratively, for $i\ge 2$ 
\begin{equation}\label{Rayleigh-var}
\L_i :=\min_{{\substack{w\in  H^1_0(B)\\ w\neq 0\\ w\perp \{ \psi_1,\dots,\psi_{i-1}\} }}} \frac{Q_a(w)}{\int_B w^2(x)\, dx} =\min_{\substack{W\subset H^1_0(B) \\ \mathrm{dim} W=i}} \max_{\substack{w\in W \\ w\neq 0}}\frac{Q_a(w)}{\int_B w^2(x)\, dx} \end{equation}  
where the 
 condition $w\perp \psi_j$ stands for the orthogonality in $L^2(B)$ {and $\psi_j$ is a function that attains $\L_j$ for $j=1,\dots,i-1$}. Again, the infimum in \eqref{Rayleigh-var} is attained at a function $\psi_i\in H^1_0(B) $ that is a weak solution  to \eqref{eigenvalue-problem} corresponding to $\L_i$.

\

Let us recall some useful properties of eigenvalues and eigenfunctions which are well known.  Even though their proofs are by now classical (see, for instance, the book \cite{CH}), we report them in view of the future extension to the singular case.

\

\noindent {\bf{Property 1.}} {\em{  The first eigenvalue $\L_1$ is simple and the first eigenfunction $\psi_1$ is strictly positive (or negative) in $B$.}}

\begin{proof}
  Substituting $\psi_1$ with $|\psi_1|$ in \eqref{Rayleigh1} one obtains that also $|\psi_1|$ attains $\L_1$ and so satisfies \eqref{eigenvalue-problem} corresponding to $\L_1$. The $C^{1,\beta}$ regularity of $\psi_1$ in Remark \ref{regolarita} then implies that $\psi_1$ does not change sign in $B$ for any first eigenfunction. Then we can assume $\psi_1\geq 0$ in $B$. 
  The strict positiveness of $\psi_1$ is due to the strong maximum principle applied to the linear operator $-\Delta -(a+\L_1)$. The fact the $\L_1$ is simple is a consequence of the positiveness of $\psi_1$. Indeed, since $L_a$  is linear if $\psi_1$ and $\psi_2$ solve \eqref{eigenvalue-problem} corresponding to $\L_1$ then also $a\psi_1+b\psi_2$ solves \eqref{eigenvalue-problem} for any $a,b\in \R$, and then minimizes \eqref{Rayleigh-var}. As before this means that also $a\psi_1+b\psi_2$ has one sign in $B$ for every $a,b\in \R$ and this is not possible.
  \end{proof}

\noindent {\bf{Property 2.}} {\em{Eigenfunctions related to different eigenvalues are orthogonal in $L^2(B)$.}}

\begin{proof}
The assertion follows from the weak formulation of \eqref{eigenvalue-problem}. Indeed assume $\psi_i$ and $\psi_j$ are eigenfunctions related respectively to $\L_i$ and $\L_j$ with $\L_i\neq \L_j$. Using $\psi_j$ as test function in the weak formulation of the equation satisfied by $\psi_i$ and using $\psi_i$ as test function in the weak formulation of the equation satisfied by $\psi_j$ we obtain 
\[\L_i\int_B \psi_i\psi_j\, dx=\int_B \nabla \psi_i\nabla \psi_j -a\psi_i\psi_j\, dx=\L_j\int_B \psi_i\psi_j\, dx\]
which implies $\displaystyle \left(\L_i-\L_j\right)\int_B \psi_i\psi_j\, dx=0$.
\end{proof}

This suggests that any eigenfunction related to an eigenvalue $\L_i>\L_1$  changes sign in $B$ (recall that $\psi_1>0$) and so it has at least two nodal domains.

\

\noindent {\bf{Property 3.}} {\em{The $i$-th eigenfunction $\psi_i$  has at most $i$ nodal domains.}}

\begin{proof} We  have already shown that this is true for $i=1$. Assume $i>1$, let $D_k$ be a nodal domain of $\psi_i$ and let $w_k$ be the function that coincides with the eigenfunction $\psi_i$ in $D_k$ and is zero elsewhere. Of course $w_k\in H^1_0(B)$ for any $k$. If the number of nodal domains of $\psi_i$ overpasses $i$, let us choose $i$ among them. It is possible to choose $i$ coefficients $c_1\dots c_i$ (not all null) such that the function $w=\sum\limits_{k=1}^ic_k w_k$ is orthogonal in $L^2(B)$ to all the previous eigenfunctions $\psi_1, \dots \psi_{i-1}$. Besides using $w_k$ as a test function  in the weak formulation of \eqref{eigenvalue-problem} one sees that 
	\begin{align*}
	\int_{D_k} \left(|\nabla w_k|^ 2 - a w_k^2\right)dx =\L_i \int_{D_k}  w_k^2 dx
	\intertext{and then, since the functions $w_k$ have disjoint supports}
	\dfrac{\mathcal{Q}_a(w)}{\int_B w^2 dx} = \dfrac{\sum\limits_{k=1}^jc^2_k\int_{D_k} \left(|\nabla w_k|^2 - a w_k^2\right) dx}{\sum\limits_{k=1}^jc^2_k\int_{D_k} w_k^2 dx} = \Lambda_i ,
	\end{align*}
namely $w$ attains the minimum in \eqref{Rayleigh-var}.
It follows that $w$ 
weakly solves \eqref{eigenvalue-problem} in $B$.
On the other hand $w$ is null in $D_{i+1}\subset B$, $D_{i+1}\neq \emptyset$ and therefore it must be identically zero by the unique continuation principle, see \cite{GL}. \end{proof}

\subsection{Eigenvalues with symmetries and Sturm-Liouville problems}\label{sse:prel-radial}
In the applications it can be useful to consider functions $a$ with some symmetries. Since we are in a radially symmetric domain if $\mathcal{G}$ is any subgroup of the orthogonal group $O(N)$  of $\R^N$ we say that a function $w(x)$ is $\mathcal{G}$-invariant if 
\[
w(g(x))=w(x) \quad \forall \ x\in B \quad \forall \ g \in \mathcal{G}.
\]
Obviously $w$ is radially symmetric if it is  $\mathcal{G}$-invariant with $\mathcal{G}=O(N)$. 
\par
When the function $a$ is $\mathcal{G}$-invariant  we can restrict the linear operator $L_a$ in \eqref{L-a} to the symmetric space 
\begin{equation}\label{H-1-G}
H^1_{0,\mathcal{G}}(B)\ :=\{ w\in H^1_0(B) \, : \, w \text{ is $\mathcal{G}$-invariant}\} ,
\end{equation}
and define the $\mathcal{G}$-invariant eigenvalues of $L_a$ ($\L_i^{\mathcal{G}}$ in the following) as in \eqref{Rayleigh-var}, but  with $H^1_{0,\mathcal{G}}(B)$ instead of $H^1_0(B)$. 
By the principle of symmetric criticality  of \cite{palais} an eigenfunction $\psi_i^{\mathcal{G}}$ which attains $\L_i^{\mathcal{G}}$ belongs to $H^1_{0, \mathcal{G}}(B)$ and weakly solves \eqref{eigenvalue-problem}. Reasoning as in the beginning of the previous Section with $H^1_{0,\mathcal{G}}(B)$ instead of $H^1_0(B)$ one sees that $L_a$ admits a sequence of eigenvalues $\L_i^{\mathcal{G}}$ such that $\L_i^{\mathcal{G}}\to +\infty$ as $i\to \infty$, and an Hilbert basis of eigenfunctions $\psi_i^{\mathcal{G}}$ that belong to $C^{1,\beta}(\bar B)$ by elliptic regularity.

The Properties 1 and 2 of the eigenvalues $\L_i$ mentioned above continue to hold for these symmetric eigenvalues $\L_i ^{\mathcal{G}}$. Moreover we have:

\

\noindent {\bf{Property 4.}} {\em{If $a$ is $\mathcal{G}$-invariant  for a  subgroup $\mathcal{G}$ of $O(N)$, then also $\psi_1$ is $\mathcal{G}$-invariant.}}

\begin{proof}
Since $\psi_1^{\mathcal{G}}>0$, and it is a positive eigenfunction of the linear operator $L_a$, then $\psi_1^{\mathcal{G}}=\psi_1$ and $\L_1=\L_i ^{\mathcal{G}}$.
\end{proof}

 We are interested in particular in the radial symmetry. So, if $\mathcal{G}=O(N)$ and hence $a$ is radial we can define 
the radial eigenvalues $\L_i^{\rad}$ of $L_a$ according to \eqref{Rayleigh-rad} where by $H^1_{0,\rad}(B)$ we denote the subspace of $H^1_0(B)$ given by radial functions and the corresponding radial eigenfunctions $\psi_i^{\rad}$, that weakly solve the Sturm-Liouville problem \eqref{radial-eigenvalue-problem}.
For what we said before $L_a$ admits a sequence of eigenvalues $\L_i^{\rad}$, such that $\L_i^{\rad}\to +\infty$ as $i\to \infty$, and an Hilbert basis of eigenfunctions $\psi_i^{\rad}$ for the subspace of $L^2(B)$ given by radial functions.

For future use we extend this notion of radial eigenvalues and eigenfunctions in \eqref{radial-eigenvalue-problem} to non integer dimension. We therefore take any $M\ge 2$ and  introduce the Lebesgue space $L^2_M$ made up by measurable functions $v:(0,1)\to\R$ such that
	\[ \int_0^1 r^{M-1} v^2 dr < +\infty.\]
	It is a  Hilbert space endowed with the product $ \langle v,w\rangle_M = \int_0^1 r^{M-1} v \, w \, dr$, 
	which yields the orthogonality condition
	\[ v \perp_M w \, \Longleftrightarrow \, \int_0^1 r^{M-1} v \, w \, dr = 0 .\]
	Next we denote by $H^1_{M}$ its subspace made up by that functions $v$ which have weak first order derivative in $L^2_M$, so that the norm 
\[ \|v\|_M = \left(\int_0^1 r^{M-1} \left( v^2+|v'|^2\right) dr\right)^{\frac{1}{2}} \]
is bounded. 
Further  by \cite[VIII.2]{Bbook}  any function in $v\in H^1_{M}$ is almost everywhere equal to a function $\tilde v\in C(0,1]$ which is differentiable almost everywhere with
\begin{equation}\label{ffci} 
\tilde v(r_2)-\tilde v(r_1) = \int_{r_1}^{r_2} v'(r) dr .\end{equation}
Therefore we may assume w.l.g.~that any $v\in H^1_{M}$ is continuous in $(0,1]$ and satisfies \eqref{ffci}.
This allows to introduce the set
\begin{equation}\label{accazeroemme}
H^1_{0,M} = \left\{ v\in H^1_M \, : \, v(1)=0\right\} .
\end{equation}
The spaces $H^1_M$ and $H^1_{0,M}$ can be seen as generalizations of the spaces of radial functions  because when $M=N$ is an integer then $H^1_N$ is actually equal to $H^1_{\rad}(B)$ by \cite[Theorem 2.2]{DFetal}. Next we  generalize the radial eigenvalues by extending
the Sturm-Lioville problem \eqref{radial-eigenvalue-problem} to
\begin{equation}\label{radial-eigenvalue-problem-M}
\left\{\begin{array}{ll}
-\left(r^{M-1}\psi_i'\right)' -r^{M-1} a(r)\psi_i =r^{M-1} \nu_i\psi_i & \text{ for } r\in ( 0,1)\\
\psi_i'(0)=0 , \quad \psi_i(1)=0  
\end{array} \right.
\end{equation}
where now $M\ge 2$ can assume any real value.
By weak solution to \eqref{radial-eigenvalue-problem-M} we mean 
\begin{equation}\label{radial-eigenvalue-weak-sol}
\begin{split}{\psi_i\in H^1_{0,M} \mbox{ such that  for every $\varphi \in H^1_{0,M}$}}\\
\int_0^1 r^{M-1}\left(\psi_i'\varphi' -  a \psi_i \varphi\right) dr = \nu_i\int_0^1 r^{M-1}\psi_i \varphi dr .
\end{split}\end{equation}
Whenever there exist $\nu_i$ and $\psi_i\neq 0$ that satisfy \eqref{radial-eigenvalue-weak-sol}, we call $\nu_i$ generalized radial eigenvalue and $\psi_i$ generalized radial eigenfunction.  Moreover when $M=N$ is an integer $\nu_i=\L_i^\rad$ are the radial eigenvalues.
  First we deal with the regularity and, in doing so, we also  point out that these generalized radial eigenfunctions enjoy the typical behaviour of the radial functions at $r=0$, i.e.~they have null derivative, although any explicit condition is not imposed at $r=0$  in the weak formulation.
 This justifies the choice of \eqref{radial-eigenvalue-weak-sol} as a definition of weak solution to \eqref{radial-eigenvalue-problem-M}. 
\begin{proposition}\label{radial-regular-0}
	If $\psi_i$ satisfies \eqref{radial-eigenvalue-weak-sol}	
		then $\psi_i\in C^1[0,1]$ with $\psi_i'(0)=0$.
\end{proposition} 
\begin{proof}
	For simplicity of notations we will write $\psi$ and $\nu$ instead of $\psi_i$ and $\nu_i$.
	Since $\psi\in H^1_{0,M}$, it is continuous on $(0,1]$ and \eqref{ffci} becomes
	\begin{equation}\label{abs-cont}
	\psi(r)= -\int_r^1 \psi'(t) dt .
	\end{equation}
		Starting from the weak formulation \eqref{radial-eigenvalue-weak-sol} it is straightforward to check that 
	\begin{equation}\label{conto-con-lo-scalino}
	t^{M-1}\psi'(t) = -\int_0^t s^{M-1} \left(a(s)+\nu\right) \psi(s) ds 
	\end{equation}
	for almost every $t\in(0,1)$.
Actually choosing a test function $\varphi_{\e}$ such that $0\le \varphi_{\e} \le 1$, $-c/\e \le \varphi'_{\e} \le 0$ and 
	\[ \varphi_{\e}(r)=\begin{cases} 1 & 0\le r \le t-\e , \\ 1-\dfrac{r-t + \e}{2\e} & |r-t|\le \e-\e^2 ,\\ 0 & t + \e \le r\le 1, \end{cases}\]		
	gives
	\begin{align*} 
	\left(\int_{t-\e}^{t-\e+\e^2} + \int_{t+\e-\e^2}^{t+\e}\right) r^{M-1}\psi'\varphi'_{\e} dr
	-
	\frac{1}{2\e}\int_{t -\e+\e^2}^{t +\e-\e^2} r^{M-1} \psi' dr = \\ \int_0^{t -\e}r^{M-1} (a+\nu) \psi dr  
	+\left(\int_{t-\e}^{t-\e+\e^2} + \int_{t+\e-\e^2}^{t+\e}\right) r^{M-1} (a +\nu) \psi\,\varphi_{\e} dr
	\\ -\frac{1}{2\e}\int_{t -\e+\e^2}^{t +\e-\e^2} r^{M-1} (r-t ) (a+\nu) \psi dr 
	+\frac{1}{2}\int_{t -\e+\e^2}^{t +\e-\e^2} r^{M-1} (a +\nu)\psi dr
	.\end{align*}
	\eqref{conto-con-lo-scalino} follows by sending $\e\to 0$. Indeed 
	\[\int_{t\pm\e}^{t\pm\e+\mp\e^2}  r^{M-1} \psi'\varphi'_{\e}  dr  \to 0 \]
	because by Holder inequality we have
	\begin{align*}
	\int_{t\pm\e}^{t\pm\e+\mp\e^2}  r^{M-1}|\psi'\varphi'_{\e}|  dr & \le \left(\int_{t\pm\e}^{t\pm\e+\mp\e^2}  r^{M-1} |\varphi'_{\e}|^2  dr\right)^{\frac{1}{2}}\left(\int_{t\pm\e}^{t\pm\e+\mp\e^2}  r^{M-1} |\psi'|^2  dr\right)^{\frac{1}{2}}
	\end{align*}
	where by the properties of $\varphi_{\e}$
	\begin{align*}
	\int_{t\pm\e}^{t\pm\e+\mp\e^2}  r^{M-1} |\varphi'_{\e}|^2  dr \le \frac{C}{\e^2} \int_{t\pm\e}^{t\pm\e+\mp\e^2}  r^{M-1}  dr \le C  
	\end{align*} 
	while clearly $\displaystyle \int_{t\pm\e}^{t\pm\e+\mp\e^2}  r^{M-1} |\psi'|^2  dr \to 0 $ since $r^{M-1} |\psi'|^2 \in L^1(0,1)$.
	Moreover it is easily seen that also the function $r^{M-1}\psi(r)$ 
	is in $L^1(0,1)$, so that the absolute continuity of the integral, together with the boundedness of $a$ and $\varphi_{\e}$,  yield 
\begin{align*}
  \int_0^{t -\e}r^{M-1} (a+\nu)\psi dr   \to \int_0^{t}r^{M-1} (a+\nu) \psi dr , \quad & \int_{t\pm\e}^{t\pm\e\mp\e^2} r^{M-1} (a+\nu) \psi\, \varphi_{\e} dr
   \to 0 , 
  \\
  \int_{t -\e+\e^2}^{t +\e-\e^2} r^{M-1} (a+\nu)\psi dr \to 0 .
\end{align*}
On the other hand, denoting by $\fint_a^b f(r)dr$ the (integral) mean value of $f$ for almost every $t$ we have
	\begin{align*}
\frac{1}{2\e}\int_{t -\e+\e^2}^{t +\e-\e^2} r^{M-1} (r-t ) (a+\nu)\psi dr
= \frac{1-\e}{2}\fint_{t -\e+\e^2}^{t +\e-\e^2} r^{M-1} (r-t ) (a +\nu)\, \psi dr \to 0, \\
	\frac{1}{2\e}\int_{t-\e+\e^2}^{t +\e-\e^2} r^{M-1} \psi' dr  = (1-\e) \fint_{t -\e+\e^2}^{t +\e-\e^2} r^{M-1} \psi' dr \to t^{M-1} \psi'(t),
	\end{align*}
as also $r^{M-1} \psi'(r)$ is in $L^1(0,1)$.
This concludes the proof of  \eqref{conto-con-lo-scalino}. 

Besides, since the right term of \eqref{conto-con-lo-scalino} is continuous, it also follows that $\psi'$ is continuous on $(0,1]$ and then  \eqref{abs-cont} implies that $\psi\in C^1(0,1]$.
If $\psi'(r)$ has a finite limit as $r\to 0$, then from \eqref{abs-cont} it follows that $\psi(r)$ is continuous  also at $r=0$ and next de L'Hopital Theorem implies that $\psi\in C^1[0,1]$.
So we can conclude the proof by checking that $\psi'(r)\to0$ as $r\to 0$. 
\\
As a preliminary step  we prove that $\psi(r)$ is bounded at $r=0$. Inserting \eqref{conto-con-lo-scalino} inside \eqref{abs-cont} gives
	\begin{equation}\label{pippo-1}
	|\psi(r)| \le (\|a\|_{\infty}+|\nu|) \int_r^1 t^{1-M} \int_0^t s^{M-1} |\psi(s)|ds \, dt ,
	\end{equation}
	which allows to start a bootstrap argument starting from the Radial Lemma \ref{radial-lemma}.
	In the following computations the symbol $C$ will denote a constant that may vary from line to line.
	To begin with we take $M=2$  and insert \eqref{radial-lemma-estimate} inside \eqref{conto-con-lo-scalino}, so that
	\[|\psi(r)| \le C \int_r^1 t^{1-M} \int_0^t s^{M-1} |\log s|^{\frac{1}{2}} ds \, dt  \le C \int_r^1 t^{1-M} \int_0^t s^{M-\frac{3}{2}}  ds \, dt = C\int_r^1 t^{\frac{1}{2}} dt  < +\infty.\]
	Otherwise if $M>2$ the same computation gives \[|\psi(r)| \le C \int_r^1 t^{1-M} \int_0^t s^{\frac{M}{2}} ds \, dt  = C \int_r^1 t^{\frac{4-M}{2}} dt.
	\]	
	If $M<6$ we the proof is ended. If $M=6$ we get $|\psi(r)|\le C|\log r|$ and putting this estimate inside \eqref{pippo-1} allows to conclude similarly to the case $M=2$. Otherwise if $M>6$ we have improved anyway the first estimate by  $|\psi(r)|\le C (1+r^{-\frac{M-6}2})$. 
	Putting this new estimate inside \eqref{pippo-1} gives
	$|\psi(r)|\le C \int_r^1(1+ t^{\frac{8-M}{2}})dt$ and iteratively $|\psi(r)|\le C \int_r^1(1+t^{\frac{4n-M}{2}})dt$ which eventually gives the statement.
	
Finally we write $\psi'(r)=\dfrac{r^{M-1}\psi'(r)}{r^{M-1}}$. 
From \eqref{conto-con-lo-scalino} it readily follows that $r^{M-1}\psi'(r)\to 0$ as $r\to 0^+$, so we can use de L'Hopital theorem to compute 	
\[\lim\limits_{r\to 0} \psi'(r)=\lim\limits_{r\to 0} \frac{r^{M-1}\psi'(r)}{r^{M-1}} = \lim\limits_{r\to 0} \frac{-r\left(a(r)+\nu\right)\psi(r)}{M-1}=0 		\]
because $\psi$ is bounded. 
\end{proof}

\begin{remark}\label{2.3}
Mimicking the computations in the concluding part of the proof of Proposition \ref{radial-regular-0} one can see that $\psi'(r)=o(r^{\gamma})$ for any $\gamma <1$. Indeed 
\[\lim\limits_{r\to 0} \frac{\psi'(r)}{r^{\gamma}}=
\lim\limits_{r\to 0} \frac{r^{M-1}\psi'(r)}{r^{M-1+\gamma}} = \lim\limits_{r\to 0} \frac{ - r^{1-\gamma}(a+\nu)\psi}{M-1+\gamma} =0.
\]
Hence $r^{M-3}|\psi'|^2\to 0$ as $r\to 0$ and so $\int_0^1 r^{M-3}|\psi'|^2 dr<+\infty$.
Besides $r^{M-1}\psi'$ has a weak derivative ${\AL -} r^{M-1} (a+\nu)\psi \in L^{\infty}(0,1)$. Next 
$\psi'' =  -(a+\nu)\psi - \frac{M-1}{r} \psi'$ on $(0,1)$ in weak sense and
\begin{align*}
\int_0^1 r^{M-1}|\psi''|^2 dr & = \int_0^1 r^{M-1}|(a+\nu)\psi - \frac{M-1}{r} \psi'|^2 dr \\
&\le C_1\int_0^1 r^{M-1} \psi^2 dr + C_2 \int_0^1 r^{M-3}|\psi'|^2 dr <+\infty,
\end{align*}
showing that $\psi'\in H^1_{M}$. 
In particular the equation in \eqref{radial-eigenvalue-problem-M} is satisfied a.e.
\end{remark}

The generalized radial eigenvalues defined by means of the singular Sturm-Liouville problem \eqref{radial-eigenvalue-problem-M} have an equivalent  variational formulation.
Indeed we can define the linear operator  $T_a : L_M^2\to L_M^2$ by means of $T_a(g)=w$  if $w\in   H^1_{0,M}$ satisfies  
	\[\int_0^1 r^{M-1}w'\varphi'\ dr-\int_0^1 r^{M-1}a(r)w \varphi \ dr=\int_0^1 r^{M-1}g \varphi \ dr  \ \ \text{ for any }\varphi \in H^1_{0,M}.\]
$T_a$ is self-adjoint  and compact because $H^1_M$ is compactly embedded in $L_M^2$ (see Lemma \ref{sobolev}).
So since $a\in L^\infty$  it admits a sequence of eigenvalues $\nu_i$ such that $\nu_i\to \infty$ as $i\to \infty$ and the corresponding eigenfunctions $\psi_i \in  H^1_{0,M}$ form a basis for $L_M^2$. In that way
\[
\nu_1 :=\min_{\substack{w\in  H^1_{0,M}\\ w\neq 0}} \frac{{\mathcal Q}_{a,M}(w)}{\int_0^1 r^{M-1} w^2(r)\, dr}  ,
\]
where  
	\begin{equation}\label{forma-quadratica-a-rad}
\mathcal Q_{a,M} : H^1_{0,M}\to \R, \qquad {\mathcal Q}_{a,M}(w)=\int_0^1 r^{M-1}\left(|w'|^2-a(r)w^2\right) dr\end{equation}
and for $i\ge 2$
	\begin{equation}\label{Rayleigh-rad-M}
\nu_i :=\min_{{\substack{w\in  H^1_{0,M}\\ w\neq 0\\ w\perp_{M} \{ \psi_1,\dots,\psi_{i-1}\} }}} \frac{Q_{a,M}(w)}{\int_0^1 r^{M-1}w^2(r)\, dr} =\min_{\substack{W\subset H^1_{0,M} \\{\mathrm{dim}} W=i}} \max_{\substack{w\in W \\ w\neq 0}}\frac{Q_{a,M}(w)}{\int_0^1 r^{M-1}w^2(r)\, dr}  .
\end{equation}
where $\psi_j$ is an eigenfunction corresponding to $\nu_j$ for $j=1,\dots, i-1$.
\\
Using this characterization it is easy to prove that Properties 1-3 hold still for these generalized eigenvalues $\nu_i$.

An exhaustive essay of classical Sturm-Liouville theory can be found, for instance, in the book \cite{Walter}, where by classical we mean for a strictly positive weight function. 
 In our case the weight $r^{M-1}$ in \eqref{radial-eigenvalue-problem-M} degenerates at the origin, nevertheless this problem inherits the main properties of the Sturm Liouville eigenvalue problems and in particular the following important one.

\

\noindent {\bf{Property 5.}} {\em{Each generalized radial eigenvalue $\nu_i$ is simple and any $i$-th eigenfunction has exactly $i$ nodal domains. In particular each radial eigenvalue $\L_i^{\rad}$ is simple and any $i$-th eigenfunction $\psi_i^{\rad}$ has exactly $i$ nodal domains.}}

\

	To prove Property 5 it is needed to make sure that the well known Sturm Separation Theorem and Sturm-Picone Comparison
	hold also in this  singular setting. This fact is a consequence of the following Picone-type identity.
\begin{lemma}[Picone identity]\label{integrali-in-croce}
		Let $\psi_i, \psi_j$ weakly solve \eqref{radial-eigenvalue-problem-M} corresponding to ${\nu}_i , {\nu}_j$ respectively. Then 	\begin{align}\label{picone-identity-0}
			\left(r^{M-1}\left(\psi_i'\psi_j- \psi_i\psi_j'\right) \right)'& = r^{M-1}(\nu_j-\nu_i) \psi_i \psi_j
		\intertext{for all $r\in[0,1]$.	If, in addition, $r$ is such that $\psi_j(r)\neq 0$, then at that $r$} \label{picone-identity}
		\left(r^{M-1}\left(\psi_i'\psi_j- \psi_i\psi_j'\right)\frac{\psi_i}{\psi_j} \right)' &= 
		r^{M-1}(\nu_j-\nu_i) \psi_i^2 + r^{M-1} \left(\psi_i' - \psi'_j \frac{\psi_i}{\psi_j}\right)^2  .
	\end{align} 
	\end{lemma}
	\begin{proof}
		By Proposition \ref{radial-regular-0} the function 
		\[ \xi(r)= r^{M-1}\left(\psi_i'\psi_j- \psi_i\psi_j'\right)\] 
		is continuous on $[0,1]$. Let us check that formula \eqref{picone-identity-0} holds in the sense of distributions. 
		Indeed for any $\varphi\in C^{1}_0(0,1)$ we have that $\varphi\psi_i, \varphi\psi_j \in C^{1}_0(0,1)$ (remembering that $\psi_i,\psi_j\in C^1[0,1]$ by Proposition \ref{radial-regular-0}) and using them as test functions in \eqref{radial-eigenvalue-problem-M} gives
		\begin{align*}
		\int_0^1 \xi\varphi' dr & = \int_0^1 r^{M-1}\psi_i'\psi_j\varphi' dr - \int_0^1 r^{M-1} \psi_j'\psi_i\varphi' dr \\
		& =\int_0^1 r^{M-1}\psi_i' \left(\psi_j\varphi\right)' dr - \int_0^1 r^{M-1} \psi_j'\left(\psi_i\varphi\right)' dr
		=(\nu_i-\nu_j)\int_0^1 r^{M-1} \psi_i\psi_j \varphi dr
		\end{align*}
                and this implies that $\xi$ has a weak derivative which is $(\nu_j-\nu_i)r^{M-1} \psi_i\psi_j$.
		So \cite[VIII.2]{Bbook} yields that
		\[
		\xi(t) -\xi(r)= (\nu_j-\nu_i)\int_r^t s^{M-1} \psi_i\psi_j ds,	
		\]
		for almost any $0<r<t\le 1$. Since both terms are continuous in $[0,1]$, the previous equality actually holds for every $0\le r < t \le 1$. Eventually \eqref{picone-identity-0} follows recalling that also the integrand in the r.A.s. is continuous, afterwards \eqref{picone-identity} is elementary calculus.
	\end{proof}

With Lemma \ref{integrali-in-croce} in hand, Property 5 can be proved reasoning as for classical eigenfunctions and any restriction does not come from the  degeneracy at $r=0$. For the sake of completeness we report here a detailed proof, which will be a guide to deal with singular eigenfunctions, in the following section.

\begin{proof}[Proof of Property 5]
First  we prove that each eigenvalue is simple. Assume by contradiction that $\nu_i=\nu_j$ for some $i\neq j$ and denote by $\psi_i$ and $\psi_j$ their respective eigenfunctions, that we can take orthogonal in the sense that $\psi_i\perp_M \psi_j$.   In particular $\psi_i$ and $\psi_j$ are linearly independent.
On the other hand \eqref{picone-identity-0} yields that $r^{M-1}(\psi_i\psi_j'-\psi_i'\psi_j)$ is constant, and therefore it must be identically zero, since both $\psi_i$ and $\psi_j$ are null at $r=1$ by definition of $H^1_{0,M}$. But this clearly yields the contradiction that the ratio $\psi_i /\psi_j$ must be constant.
	
Next  we check that the eigenfunction $\psi_{i}$ has exactly $i$ nodal domains. 
Notice that by Property 3 in the space $H^1_{0,M}$ instead of $H^1_{0}(B)$ any generalized radial eigenfunction $\psi_i$ has at most $i$-nodal domains. In particular $\psi_1$ is positive and so it has $1$ nodal domain while $\psi_2$ changes sign and so it has $2$ nodal domains.  The proof can be completed by taking that $\psi_i$ has exactly $i$ nodal zones and showing that $\psi_{i+1}$ has one nodal domain more than $\psi_i$, for $i\ge 2$. 
\\
Let $0<s_1<\dots s_{i-1}<1$ be the internal zeros of $\psi_i$ and  $s_i=1$ such that $\psi_i(s_k)=0$ for $k=1,\dots,i$. To begin with we show that $\psi_{i+1}$ vanishes at some point inside $(s_{k-1},s_k)$ for $k=2,\dots,i$. To fix idea we take that $\psi_i(r)>0$ on $(s_{k-1},s_k)$, which also implies $\psi_i'(s_{k-1})>0$ and $\psi_i'(s_{k})<0$. If $\psi_{i+1}$ does not vanishes inside $(s_{k-1},s_k)$ we may assume without loss of generality that 
$\psi_{i+1}(r)>0$ in $(s_{k-1},s_k)$ and $\psi_{i+1}(s_{k-1}), \psi_{i+1}(s_{k})\ge 0$. Integrating \eqref{picone-identity-0} on $(s_{k-1},s_k)$  gives 
\[ s_k^{M-1}\psi_i'(s_k)\psi_{i+1}(s_k)- s_{k-1}^{M-1}  \psi'_i(s_{k-1})\psi_{i+1}(s_{k-1}) = (\nu_{i+1}-\nu_i)\int_{s_{k-1}}^{s_k} r^{M-1} \psi_i \psi_{i+1} dr. \]
But this is not possible because the l.A.s. is less or equal than zero  by the just made considerations, while the r.A.s. is strictly positive as $\nu_{i+1}>\nu_i$ by the first part of the proof.
\\
The same computation also shows that $\psi_{i+1}$ has a zero also in $(0,s_1)$, since $r^{M-1}(\psi_i'\psi_{i+1}-\psi_i\psi'_{i+1})$ clearly vanishes at $r=0$ by Proposition \ref{radial-regular-0}. In this way we have proved that $\psi_{i+1}$ has at least $i$ zeros in $(0,1)$ which gives $i+1$ nodal zones as desired.
\end{proof}

\section{The singular eigenvalue problem}\label{se:singular}

Now we address to a singular eigenvalue problem for the linear operator $L_a$. Letting $a\in L^{\infty}(B)$ as before, we consider the problem  
\begin{equation}\label{singular-eigenvalue-problem}
\left\{\begin{array}{ll}
-\Delta  \psi_i-a\psi_i =\frac{\widehat\L_i}{|x|^2}\psi_i & \text{ in } B\setminus\{0\}\\
\psi_i= 0 & \text{ on } \partial B,
\end{array} \right.
\end{equation}
which is not well defined in the space $H^1_0(B)$, in general, because of the singularity at the origin.
		It is therefore needed to 
		introduce the Lebesgue space
\begin{equation}\label{L2-pesato}
{\mathcal L}:\ \{w: \ B\to \R,  \text{ measurable and s.t } \int_B |x|^{-2}w^2\, dx<\infty\},
\end{equation}
which is a Hilbert space with the scalar product $\int_B |x|^{-2}\eta\varphi\ dx$, so that 
\begin{equation}\label{scalar-H}
\eta\underline{\perp}\varphi\  \ \Longleftrightarrow \int_B|x|^{-2}\eta\varphi\ dx=0  \  \ \text{ for }\eta, \varphi\in  {\mathcal L}.
\end{equation}
Next
\begin{equation}\label{space-H}
\mathcal{H}\ :=H^1(B)\cap{\mathcal L} \ \ \text{ and }\ \ \mathcal{H}_0\ :=H_0^1(B)\cap {\mathcal L}.
\end{equation}
are Banach spaces with the norm $\norm{w}_{\mathcal{H}}^2=\int _B |\nabla w|^2 + |x|^{-2}w^2\ dx$.
In dimension $N\geq 3$ the space $\mathcal{H}_0$ coincides with $H^1_0(B)$ due to the Hardy inequality, while in dimension $2$ the space $\mathcal{H}_0$ is strictly contained in $H^1_0(B)$, consider for example the function $w(x)=1-|x|^2$. Indeed in dimension 2 every continuous function $w\in \mathcal{H}$ satisfies $w(0)=0$ since $|x|^{-2}\notin L^1(B)$. 

Eventually by weak solution to \eqref{singular-eigenvalue-problem} we mean $\psi_i \in \mathcal{H}_0$ such that 
		\begin{equation}\label{weak-equation}
		\int_B\nabla  \psi_i\nabla \varphi\ dx-\int_B a \psi_i\varphi\ dx=\widehat \L_i\int_B \frac{ \psi_i\varphi}{|x|^2}\ dx
		\end{equation}
		for every  $\varphi\in \mathcal{H}_0$.
		We call $\widehat\L_i$ eigenvalue of \eqref{singular-eigenvalue-problem} if there exists a nontrivial function $\psi_i$ that satisfies \eqref{weak-equation}, next such $\psi_i$ will be called eigenfunction.

Analogously to the classical case one can try to produce an eigenvalue by minimizing the quotient 
\begin{equation}\label{first-singular}
\widehat \L_1:=\inf_{\substack{w\in\mathcal{H}_0 \\ w\neq 0}}\frac{Q_a(w)}{\int_B |x|^{-2}w^2(x)\, dx} 
\end{equation}
This method can fail: for instance when $a\equiv 0$ the infimum in \eqref{first-singular} is $ \left(\frac{N-2}2\right)^2$ and it is not attained due to the lack of compactness of the embedding of $H^1_0(B)$
into
${\mathcal L}$, see \cite[Proposition 2.2]{GGN2}. Anyway as soon as the infimum in \eqref{first-singular} is strictly less than  $\left(\frac{N-2}2\right)^2$ 
the compactness is restored and  we have

\begin{proposition}\label{p-lambda-1}
When $\widehat \L_1<\left(\frac{N-2}2\right)^2$ then $\widehat \L_1$ is attained at a function $ \psi_1\in \mathcal H_0$ which satisfies \eqref{weak-equation} and therefore is a weak solution to \eqref{singular-eigenvalue-problem}.
Next  $ \psi_1\in C^{1,\gamma}_{\loc}(\bar B\setminus\{0\})$ for some $0<\gamma<1$ and it is a classical solution to \eqref{singular-eigenvalue-problem} in $B\setminus\{0\}$ when $a\in C^{0,\beta}(B)$ for some $\beta>0$.
\end{proposition}
This statement has been proved for $N=2$ in \cite[Proposition 2.1]{GGN2}, and for $N\ge 3$ and $\Lambda_1<0$ in \cite[Proposition 5.4]{DGG}. Here we report a unified proof which is valid also for $0\le \Lambda_1 < \left(\frac{N-2}2\right)^2$.
\begin{proof}
Let us consider a minimizing sequence $w_n\in \mathcal H_0$, normalized such that
\[\int_{B} w_n^2=1   \]
By definition
\begin{equation}\label{def-minimizzante}\int_{B}|\nabla w_n|^2-aw_n^2=\widehat \beta_n\int_{B}|x|^{-2}w_n^2\end{equation}
with $\widehat \beta_n \searrow \widehat \L_1$ as $n\to\infty$. 
We first check that
\begin{equation}\label{grad-limitato}
\int_{B}|\nabla w_n|^2\leq C.
\end{equation}
Actually if $\widehat \beta_n \le 0$ then \eqref{def-minimizzante} immediately yields 
\begin{align*}
\int_{B}|\nabla w_n|^2 \le \int_{B}aw_n^2 \leq \|a\|_{\infty} \int_{B}w_n^2  =C . 
\end{align*}
Otherwise if $\widehat \beta_n \ge 0$ definitely it means that $N\ge 3$ and we can use the Hardy inequality to get 
\[\int_{B}|\nabla w_n|^2-aw_n^2\leq \widehat \beta_n\frac 4{(N-2)^2}\int_{B}|\nabla w_n|^2 \]
which implies 
\[
\big( 1- \widehat \beta_n\frac 4{(N-2)^2}  \big)\int_{B}|\nabla w_n|^2\leq   \int_{B}aw_n^2\leq \|a\|_{\infty} \int_{B}w_n^2=C \]
and then \eqref{grad-limitato} follows because $\widehat \beta_n<\left(\frac{N-2}{2}\right)^2$. 
By \eqref{grad-limitato} it follows that, up to a subsequence, $w_n \rightharpoonup w$ weakly in $H^1_0(B)$ and strongly in $L^2(B)$, in particular
\begin{equation}
\label{pippo} \lim\limits_{n\to\infty} \int_B a w_n^2 = \int_{B}aw^2 .
\end{equation}
Moreover 
\begin{equation}	\label{pluto} \int_{B} |x|^{-2} w_n^2 \le C .	\end{equation}	
Actually if $\Lambda_1 \neq 0 $ this follows because by \eqref{def-minimizzante}
\[
\int_{B}|x|^{-2}w_n^2 = \frac{1}{\widehat\beta_n }\int_{B}|\nabla w_n|^2-aw_n^2 \le \frac{1}{|\widehat\beta_n|}\int_{B}|\nabla w_n|^2 + \frac{\|a\|_{\infty}}{|\widehat\beta_n|}\int_{B} w_n^2 \le C
\]
by \eqref{grad-limitato}, \eqref{pippo}.
Otherwise if $\Lambda_1=0$, then necessarily $N\ge 3$ and  Hardy inequality implies that
\[\int_{B}|x|^{-2}w_n^2\leq \frac 4{(N-2)^2}\int_{B}|\nabla w_n|^2\leq C\]
by \eqref{grad-limitato}.

Next we check that $w$ minimizes the quotient in \eqref{first-singular}, namely
\begin{equation}
\label{paperino}  \int_{B}|\nabla w|^2-aw^2 - \widehat \L_1\int_{B}|x|^{-2}w^2\leq 0 .
\end{equation}
When $\widehat \L_1 <0$ it suffices to apply Fatou's Lemma and get 
\[\begin{split}
&    \int_{B}|\nabla w|^2-aw^2 - \widehat \L_1\int_{B}|x|^{-2}w^2\leq \liminf_{n\to \infty} \Big(\int_B|\nabla w_n|^2- \beta_n |x|^{-2} w_n^2 \Big)- \int_{B}aw^2\\
&= \liminf_{n\to \infty} \int_{B}a w_n^2- \int_{B}aw^2= 0
\end{split}
\]
by \eqref{pippo}. Otherwise, when $\widehat \L_1 \ge 0$ (and therefore $N\ge 3$) we have that
\[\begin{split}
\int_{B}|\nabla w|^2-aw^2 - \widehat \L_1\int_{B}|x|^{-2}w^2\underset{\text{Fatou and \eqref{pippo}}}{\leq}
\\ \liminf_{n\to \infty} \Big(\int_B|\nabla w_n|^2- a w_n^2 \Big)-\widehat \L_1 \int_{B}|x|^{-2}w^2
= \liminf_{n\to \infty} \beta_n\int_{B}|x|^{-2}w_n^2-\widehat \L_1 \int_{B}|x|^{-2}w^2 \\ \le  
\liminf_{n\to \infty} (\beta_n-\widehat \L_1)\int_{B}|x|^{-2}w_n^2
+\widehat \L_1\, \Big(\liminf_{n\to \infty} \int_{B}|x|^{-2}w_n^2 -  \int_{B}|x|^{-2}w^2 \Big) \\
\underset{\text{Fatou}}{\leq}\liminf_{n\to \infty} (\beta_n-\widehat \L_1)\int_{B}|x|^{-2}w_n^2 \underset{\eqref{pluto}}{\leq} C  \lim_{n\to \infty} (\beta_n-\widehat \L_1) =0 .
\end{split}
\]
Eventually, as $w$ minimizes  the quotient in \eqref{first-singular}, for any $\phi\in \mathcal{H}_0$ the function
\[F(t):=\frac{Q_a(w+t \phi  )}{\int_B |x|^{-2}  \left( w+t \phi\right)^2 \, dx}\]
has a minimum at $t=0$ and this implies that $w=\psi_1$ is a weak solution to \eqref{weak-equation}.

Next standard elliptic estimates give the stated regularity as in Remark \ref{regolarita}. \end{proof}

When $\widehat \L_1<\left(\frac{N-2}2\right)^2$ and it is attained at a function $\psi_1\in \mathcal{H}_0$ we can define
\[
\widehat \L_2:=\inf_{\substack{w\in\mathcal{H}_0\ w\neq 0 \\ w\underline{\perp} \psi_1}}\frac{Q_a(w)}{\int_B |x|^{-2}w^2(x)\, dx}.  
\]
Iteratively if $\widehat \L_{i-1}<\left(\frac{N-2}{2}\right)^2$ and $\psi_{i-1}$ is an eigenfunction that attains it, we can define
\begin{equation}\label{i-singular}
\widehat \L_{i}:=\inf_{\substack{\ \ w\in\mathcal{H}_0\; w\neq 0 \\{ w \underline{\perp} \{\psi_1,\dots,\psi_{i-1}\} }}}\frac{Q_a(w)}{\int_B |x|^{-2}w^2(x)\, dx} 
\end{equation}
These numbers $\widehat \L_{i}<\left(\frac{N-2}{2}\right)^2$ are eigenvalues of \eqref{singular-eigenvalue-problem} indeed we have 

\begin{proposition}\label{prop-autov-successivi}
Whenever the value $\widehat \L_{i}$ defined in \eqref{i+1-singular} satisfies $\widehat \L_{i}<\left(\frac{N-2}{2}\right)^2$ then it is attained at a function $ \psi_{i}\in \mathcal{H}_0$, which is a weak solution to \eqref{singular-eigenvalue-problem} corresponding to $\widehat \L_{i}$, satisfies $\psi_{i}\underline{\perp}\{\psi_1,\dots,\psi_{i-1}\}$ and has the same regularity as $\psi_1$.
\end{proposition}
\begin{proof}
Reasoning as in the proof of Proposition \ref{p-lambda-1} one can see that a minimizing sequence $w_n$ converges to a function $w$  weakly in $H^1_0(B)$, strongly in $L^2(B)$ and pointwise a.e. (up to a subsequence). Moreover \eqref{pluto} assures that there exists a subsequence of
$w_n$, that we denote $w_{n_k}$ that converges to $w$ also weakly in ${\mathcal L}$, via the Banach-Alaoglu Theorem.  Then we have \[0=\lim_{k\to \infty}\int_B |x|^{-2}w_{n_k} \psi_j=\int_B |x|^{-2}w\psi_j\]
for $j=1,\dots,i-1$, meaning that $w\underline{\perp}\{\psi_1,\dots,\psi_{i-1}\}$. Then, as before it follows that $   \widehat\Lambda_i$ is attained and $w=\psi_i$ is a weak solution to \eqref{singular-eigenvalue-problem} corresponding to $\widehat \L_i$ in the sense of \eqref{weak-equation}.
\end{proof}

\noindent By Proposition \ref{prop-autov-successivi} the numbers  $ \widehat \L_{i}$ defined in \eqref{i+1-singular} are eigenvalues of problem \eqref{singular-eigenvalue-problem} in the sense that there exists a weak solution to \eqref{singular-eigenvalue-problem} corresponding to $ \widehat \L_{i}$ when they satisfy $\widehat \L_{i}< \left(\frac{N-2}{2}\right)^2$. Also the converse is true indeed it is easy to prove that 

\begin{lemma}[Variational characterization] \label{lem-var} 
	The eigenvalues of problem  \eqref{singular-eigenvalue-problem}
	which are less than $\left(\frac{N-2}{2}\right)^2$ 
	coincide with the numbers  $\widehat \L_i$  defined in \eqref{i+1-singular}.
\end{lemma}

Moreover these eigenvalues $\widehat \L_{i}$ satisfy: 

\

\noindent {\bf{Property 1.}} {\em{
	The first eigenvalue $\widehat \L_1$ is simple and the first eigenfunction $ \psi_1$ is strictly positive (or negative) in $B\setminus\{0\}$.}}

\

\noindent {\bf{Property 2.}} {\em{Eigenfunctions related to different eigenvalues $\widehat \L_i$ are orthogonal in $\mathcal L$.}}

\

This implies that any eigenfunction related to an eigenvalue $\widehat\L_i>\widehat\L_1$  changes sign in $B$ (recall that $ \psi_1>0$ in $B\setminus\{0\}$). Next

\

\noindent {\bf{Property 3.}} {\em{The $i$-th eigenfunction $ \psi_i$  has at most $i$ nodal domains.}}

\

The proof of these properties readily follows as in the case of the standard eigenvalues which is reported in Section \ref{se-prel}.

 The following relation links the negative eigenvalues of \eqref{eigenvalue-problem} and \eqref{singular-eigenvalue-problem}; we will use it in the sequel in order to characterize the Morse index of solutions to problem \eqref{general-f}.

\begin{proposition}\label{prop-prel-1}
The number of negative eigenvalues $\L_i$ according to \eqref{Rayleigh-var}  coincides with the number of negative eigenvalues $ \widehat \L_{i}$ defined in \eqref{i+1-singular} corresponding to the same  $a(x)$.
\end{proposition}
\begin{proof}
This result is proved in Lemma 2.6 in \cite{GGN2} for the case of dimension $2$ and in Lemma 5.6 in \cite{DGG} for higher dimensions and uses the fact that the quadratic form $\mathcal Q_a$ in \eqref{forma-quadratica-a} is negative defined both in the space spanned by the singular eigenfunctions related to negative singular eigenvalues $\widehat \psi_1,\dots,\widehat \psi_K$ or in the space spanned by the regular eigenfunction $ \psi_1,\dots,\psi_H$ related to negative eigenvalues. This imply that $H=K$.
\end{proof}

Observe that this result is true as far as one consider the eigenvalues which are strictly negative and cannot be true for the zero eigenvalue as observed in  \cite[Lemma 3.6]{GI} in the case of dimension 2.
\begin{remark}\label{faggio}Proposition \ref{prop-prel-1} together with the
variational characterization of the standard eigenvalues $\L_i$ in \eqref{Rayleigh-var} imply that the number of negative singular eigenvalues $\widehat \L_i$ coincides with the maximal dimension of a subspace of $H^1_0(B)$ in which the quadratic form $\mathcal Q_a$  in \eqref{forma-quadratica-a} is negative defined. \end{remark}
In dimension $N\ge 3$ a similar relation holds also for the eigenvalue zero. This fact will be useful to characterize the degeneracy of a solution to a semilinear PDE.
\begin{proposition}\label{prop-prel-0}
When $N\ge 3$, $\Lambda_i=0$ if and only if $ \widehat \L_{i}=0$ for the same index $i$.
\end{proposition}
\begin{proof}
Let us first remark that when $N\ge 3$ ${\mathcal H}_0= H^1_0(B)$ by Hardy inequality. 
So if $\Lambda_1=0$ (respectively, $\widehat{\L}_1=0$), then 
\[\inf\limits_{H^1_0(B)} {\mathcal Q}_a =0 ,\]
which also means that $\widehat\L_1=0$ (resp., $\Lambda_1=0$), so that both $\L_1=0$ and $\widehat\L_1=0$.
\\	 If $i\ge 2$, $\L_i=0$  and $\L_{i-1}<0$,  we denote by $\psi_i$ the eigenfunction realizing the minimum in \eqref{Rayleigh-var} for $\Lambda_i$  and by $\widehat W$ the subspace generated by the first $i-1$ eigenfunctions of \eqref{i+1-singular}, that we denote by now as  $\widehat\psi_i, \dots, \widehat\psi_{i-1}$, and  that we can choose to be orthogonal in ${\mathcal L}$.
By the variational characterization of the eigenfunctions 
\begin{equation}\label{auto-0-weak} \int_B \nabla \psi_i \, \nabla \varphi dx - \int_B a\,  \psi_i \, \varphi dx =0 
\end{equation}	for any $\varphi\in H^1_0(B)$. 
Moreover the quadratic form $\mathcal{Q}_a$ is negative defined on $\widehat W$. Indeed if $\widehat \psi=\sum_{j=1}^{i-1}a_j \widehat\psi_j$ then
	\[\begin{split}
	Q_a(\widehat \psi)&=\int_B |\nabla \widehat \psi|^2-a\widehat \psi^2
	= \sum_{j=1}^{i-1}a_j^2 \int_B |\nabla  \widehat\psi_j|^2-a \widehat\psi_j^2+\\
	&\sum_{j,k=1 \ j\neq k}^{i-1}a_ja_k \int_B \nabla \widehat\psi_j  \nabla  \widehat\psi_k -a \widehat\psi_j \widehat\psi_k
	\end{split}\]
	Using the weak formulation \eqref{weak-equation} we have
	\[\begin{split}
	Q_a(\widehat \psi)&=\sum_{j=1}^{i-1}a_j^2\widehat\Lambda_j\int_B|x|^{-2} \widehat\psi_j^2+\sum_{j,k=1 \ j\neq k}^{i-1}a_ja_k \widehat\Lambda_j\int_B|x|^{-2} \widehat\psi_j \widehat\psi_k\\
	&=\sum_{j=1}^{i-1}a_j^2\widehat\Lambda_j\int_B|x|^{-2} \widehat\psi_j^2<0
	\end{split}\]
	by the orthogonality conditions and since $ \widehat\Lambda_j<0$ for any $j$ by Proposition \ref{prop-prel-1}. 
Further, again by Proposition \ref{prop-prel-1}
\[\inf\left\{ \mathcal{Q}_a(w) \, : \,  w\in\mathcal{H}_0 , \,  w\neq 0 \, \,  w\underline{\perp} \widehat W \right\} \ge 0 = \mathcal{Q}_a(\psi_i). \]
If $\psi_i\underline{\perp} \widehat W$ it follows that $\psi_i$ realizes the minimum and therefore it is also an eigenfunction for $\widehat{\L}_i=0$.
On the other hand if $\psi_i$ is not orthogonal to $\widehat W$ (according to \eqref{scalar-H}) we could write $\psi_i= v +  w$ with $v \in \widehat W \setminus\{0\}$ and $w \underline{\perp} \widehat W$, so that $\mathcal{Q}_a(v)< 0$ and $\mathcal{Q}_a( w)\ge 0$.
On the contrary $\mathcal{Q}_a(v) = \mathcal{Q}_a( w)$ because 
\begin{align*} 
\mathcal{Q}_a(v) &= \mathcal{Q}_a(\psi_i) + \mathcal{Q}_a(w) - 2 \left(\int_B \nabla \psi_i \nabla  w dx - \int_B a  \psi_i w dx\right) 
\end{align*}
where one can see that both $\mathcal{Q}_a(\psi_i)$ and $\int_B \nabla \psi_i \nabla w dx - \int_B a  \psi_i  w dx$ are null by using respectively $\psi_i$ and $w$ as a test function in \eqref{auto-0-weak}. This proves $\psi_i \underline{\perp} \widehat W$ and $\widehat \Lambda_i=0$.
\\
If also $\L_{i+1}=0$, let $\psi_{i+1}$ the solution to \eqref{auto-0-weak} realizing the minimum in \eqref{Rayleigh-var} for $\Lambda_{i+1}$. One can check as in the previous case that $\psi_{i+1}\underline{\perp}\widehat W$, and by construction $\psi_{i+1}\perp \psi_i$, but possibly 
\[ k = \int_B \frac{\psi_{i} \psi_{i+1}}{|x|^2} dx \neq 0 . \]
We therefore set $\widetilde \psi_{i+1} = \psi_{i+1} - k \psi_i$. It is clear that $\widetilde\psi_{i+1} \underline{\perp} \psi_i$, and also that $\widetilde\psi_{i+1} \underline{\perp} \widehat W$ (because both $\psi_{i}$ and $\psi_{i+1}$ are $\underline{\perp}\widehat W$). On the other hand by  linearity $\widetilde \psi_{i+1}$ solves \eqref{auto-0-weak}, and so it realizes the minimum in \eqref{i+1-singular} assuring that $\widehat \L_{i+1}=0$.
\\
We therefore see that  $\widehat \L_{j}=0$ whenever $j$ is such that $\L_{j}=0$.
Similarly one can check that $\L_{j}=0$ whenever $j$ is such that $\widehat \L_{j}=0$: it suffices to switch the symbols $\perp$ and $\underline{\perp}$ (and the respective scalar product and projection).
\end{proof}

\subsection{Eigenvalues in symmetric spaces and singular Sturm-Liouville problems}\label{sse:general}
As we did in Subsection \ref{sse:prel-radial},  when $a$ possess some symmetries,
we can consider singular eigenvalues with symmetries restricting the definition of $ \widehat \L_{i}$ in \eqref{i+1-singular} to the spaces $\mathcal{H}\cap H^1_{0,{\mathcal{G}}}(B)$ (see \eqref{H-1-G}) for some subgroup $\mathcal{G}$ of the orthogonal group $O(N)$. The analogous of Propositions \ref{p-lambda-1}, \ref{prop-autov-successivi} and Lemma \ref{lem-var}	hold for these symmetric singular eigenvalues $ \widehat \L_{i}^\mathcal{G}$, so that they are attained whenever $ \widehat \L_{i}^\mathcal{G}<\left(\frac{N-2}2\right)^2$, the corresponding eigenfunction $ \psi_i^\mathcal{G}$ is a weak solution to \eqref{singular-eigenvalue-problem} and belong to $C^{1,\gamma}_{\loc}(\bar B\setminus\{0\})$, and conversely if \eqref{singular-eigenvalue-problem} has a nontrivial weak solution in $\mathcal{H}\cap H^1_{0,{\mathcal{G}}}(B)$	for some  $ \widehat \L_{i}^\mathcal{G}<\left(\frac{N-2}2\right)^2$, then  $\widehat \L_i^\mathcal{G}$ is an eigenvalue according to the definition  \eqref{i+1-singular} (restricted to $\mathcal{H}\cap H^1_{0,{\mathcal{G}}}(B)$). 
In particular, when $\mathcal{G}=O(N)$ and $a$ is radial, we let 
\[\mathcal{H}_{0,\rad}:=  \mathcal{H}\cap H^1_{0,\rad}(B)\]
and  write $\widehat\Lambda_i^{\rad}< \left(\frac{N-2}2\right)^2$, $ \psi_i^{\rad} \in \mathcal{H}_{0,\rad}$  for the respective eigenvalues and eigenfunctions.
\\
The Properties 1 and 2  are still satisfied, moreover

\ 

\noindent {\bf{Property 4.}} {\em{If $a$ is $\mathcal{G}$-invariant with respect to some subgroup $\mathcal{G}$ of $O(N)$ and $\widehat \L_1^{\mathcal G}<\left(\frac{N-2}2\right)^2$, then $\widehat \L_1=\widehat \L_1^{\mathcal G}$ and the first eigenfunction	$ \psi_1$ is $\mathcal{G}$-invariant. In particular if  $\widehat \L_1^{\rad}<\left(\frac{N-2}2\right)^2$, then $\widehat \L_1=\widehat \L_1^{\rad}$ and the first eigenfunction $ \psi_1$ is radial.}}

\

The proof follows exactly as in Subsection \ref{sse:prel-radial} and we omit it.

Finally we can extend to these eigenvalues with symmetry also Proposition \ref{prop-prel-1} and Proposition \ref{prop-prel-0} getting:
\begin{proposition}\label{prop-autov-rad}
The number of negative eigenvalues with symmetry $\Lambda_i^ \mathcal{G}$, according to their definition in Subsection \ref{sse:prel-radial}, coincides with the number of negative singular eigenvalues with symmetry $\widehat\Lambda_i^ \mathcal{G}$. Moreover, when $N\geq 3$, $\Lambda_i^ \mathcal{G}=0$ if and only if $\widehat \Lambda_i^ \mathcal{G}=0$, for the same index.	\\
	In particular 	the number of negative radial eigenvalues $\Lambda_i^\rad$ coincides with the number of negative radial singular eigenvalues $\widehat\Lambda_i^ \rad$, and   when $N\geq 3$ $\Lambda_i^ \rad=0$ if and only if $\widehat \Lambda_i^ \rad=0$, for the same index.
\end{proposition}

For future use we broaden the notion of radial singular eigenvalues, as we did in Subsection \ref{sse:prel-radial} for the standard radial eigenvalues. 
With this aim, for any $M\in\R$, $M\ge 2$, we define the Lebesgue  space
\begin{equation}\label{L2-M-pesato}
{\mathcal L}_M\; :\ \{w: \ (0,1)\to \R,  \text{ measurable and s.t } \int_0^1 r^{M-3}w^2\, dr<\infty\}
\end{equation}
with the scalar product $\int_0^1 r^{M-3}\eta\varphi\ dr$, so that 
\begin{equation}\label{scalar-HM}
\eta\underline \perp_{M}\varphi\  \ \Longleftrightarrow   \int_0^1 r^{M-3}\eta\varphi\ dr=0  \  \ \text{ for }\eta, \varphi\in {\mathcal L}_M.
\end{equation}
Then we define the Banach spaces
\begin{equation}\label{H-m}
\mathcal{H}_M\ :=H^1_M\cap {\mathcal L}_M\ \ \text{ and }\ \ \mathcal{H}_{0,M}\ :=H^1_{0,M}\cap  {\mathcal L}_M ,
\end{equation}
where $H^1_{M}$ and $H^1_{0,M}$ have been defined  in \eqref{ffci} and \eqref{accazeroemme}.

For  any $a\in L^{\infty}(0,1)$  we look at the weak formulation of the singular Sturm-Liouville problem 
	\begin{equation}\label{radial-singular-problem-M}
	\left\{\begin{array}{ll}
	- \left(r^{M-1} \psi'\right)'- r^{M-1} a\, \psi = r^{M-3} \widehat{\nu}_i  \psi & \text{ for } r\in(0,1)\\
	\psi\in  \mathcal H_{0,M}
	\end{array} \right.
\end{equation} 
	with $\widehat{\nu}_i\in \R$.  
	A weak solution to \eqref{radial-singular-problem-M} is $\psi\in \mathcal{H}_{0,M}$ such that 
	\begin{equation}\label{weak-radial-general} 
	\int_0^1 r^{M-1}\left( \psi_i'\varphi' - a\,\psi_i \varphi\right) dr =\widehat{\nu}_i\int_0^1 r^{M-3} \psi_i\varphi\, dr \end{equation}
	for any $\varphi\in \mathcal{H}_{0,M}$.
We say that $\widehat\nu_i$ is a generalized radial singular eigenvalue if there exists $\psi_i\in \mathcal{H}_{0,M}\setminus\{0\}$ that satisfies \eqref{weak-radial-general}. Such $\psi_i$ will be called a generalized radial singular eigenfunction  because when $M=N$ is an integer then $\widehat{\nu}_i=\widehat \L^{\rad}_i$ are the radial singular eigenvalues according to the previous definition, by the equivalence between $H^1_N$ and $H^1_{\rad}(B)$.
\\

{
It is worthwhile to spend few words describing the regularity and the  behaviour of the generalized radial singular eigenfunction near at $r=0$. 
\begin{proposition}\label{prop-prel-3}
	Let
	$\psi$ be a weak solution to  \eqref{radial-singular-problem-M} with
	$M\geq 2$ and $\widehat{\nu} < 0$. 
	Then $\psi\in C[0,1]\cap C^1(0,1]$  and, letting $\theta=\frac{2-M+\sqrt{(M-2)^2-4\widehat{\nu}}}{2}$, it satisfies
	\begin{equation}\label{decay-psi}
	\psi(r)=O\left(r^{\theta}\right)  \ \ \text{ and  } \ \ \ \psi'(r)=O\left(r^{\theta-1}\right)
	\ \ \text{ as }r\to 0
	\end{equation}
\end{proposition} 
\begin{proof}
{The first estimate in \eqref{decay-psi} has been proved for classical solutions in \cite{GGN} and \cite{GI} in the case $M=2$ and in Lemma 5.9 in \cite{DGG} in the case $M>2$. In both cases it is a consequence of the  local estimate
	\begin{equation}\label{stima-ggn}
	r^{M-1}\psi'(r)=\left\{
	\begin{array} {ll} O(1) & \text{ when }M>2\\
	O(\log r)& \text{ when }M=2
	\end{array}\right.
	\ \ \text{ as }r\to 0, 
	\end{equation}
	which is obtained directly from \eqref{radial-singular-problem-M} integrating on $(r,1)$. 
	Let us sketch here how to  obtain  \eqref{stima-ggn} for weak solutions.
	For every $ a,b\in [0,1]$, $a<b$  and for every $\e>0$ sufficiently small, let us consider a function $\varphi_{\e, a,b}\in C^{\infty}(0,1)$ such that 
	\begin{equation}\label{gradino}
	\varphi_{\e, a,b}:=\begin{cases}
	0 & \text{ if }  0<r<a+\e , \ b-\e<r\leq 1   \\
	1 & \text{ if }  a+2\e<r<b-2\e, 
	\end{cases}\end{equation}
	with $0\leq  \varphi_{\e,a,b}\leq 1  $ and $|\varphi_{\e,a,b}'(r)|\leq \frac C \e$, see Proposition \ref{radial-regular-0}. By definition $\varphi_{\e, a,b} \in C^{\infty}_0(0,1)$ for every $\e$ small enough 
	\\
	Since $\psi\in H^1_M$, the function $r^{M-1}\psi'$ belongs to $L^1(0,1)$ and therefore 	almost every $r\in (0,1)$ is a Lebesgue point for $r^{M-1}\psi'$.
	Let $t\in (0,1)$ be one of such Lebesgue point fixed once and for all. If  $\bar r\in (0,t)$ is another Lebesgue point of $r^{M-1}\psi'$, we choose $a=\bar r$, $b=t$ and use $\varphi_\e:=\varphi_{\e, \bar r,t}$  as test function in \eqref{weak-radial-general}, getting
	\[\int_0^1 r^{M-1}\psi'\varphi_\e' \, dr=\int_0^1 r^{M-1}\left(a(r)+\frac  {\widehat{\nu}}{r^2}\right)\psi \varphi_\e\, dr .
	\]
	Observe then that 
	\[\int_0^1 r^{M-1}a(r)\psi \varphi_\e\, dr=\int_{\bar r+\e}^{t-\e}r^{M-1}a(r)\psi \varphi_\e\, dr\]
	and that in the interval $(\bar r+\e,t-\e)$
	\[ |a(r)\psi \varphi_\e|\leq |a(r)\psi|\leq C\sup_{\left[\frac {\bar r}2, 1\right]}|\psi(r)|\leq C\]
	for a constant $C$ that depends on $\bar r$ but is independent on $\e$ as $\psi\in H^1_M\subset C(0,1]$. We can then pass to the limit in the above integral as $\e\to 0$ getting that
	\[\int_0^1 r^{M-1}a(r)\psi \varphi_\e\, dr\to \int_{\bar r}^t r^{M-1}a(r)\psi\, dr.\]
	Moreover, in the same way, since $|r^{M-3}\psi\varphi_\e|\leq |r^{M-3}\psi|$ and $r^{M-3}\psi\in L^1(\frac {\bar r}2,1)$ we have that
	\[\int_0^1 r^{M-3}\psi \varphi_\e\, dr=\int_{\bar r+\e}^{t-\e}r^{M-3}\psi \varphi_\e
	\to \int_{\bar r}^t r^{M-3}\psi\, dr \]
	Finally 
	\[\int_0^1 r^{M-1}\psi'\varphi_\e' \, dr=\int_{\bar r+\e}^{\bar r+2\e}r^{M-1}\psi'\varphi_\e' \, dr+\int_{t-2\e}^{t-\e}r^{M-1}\psi'\varphi_\e'\, dr\]
	by definition of $\varphi_\e$. Passing to the limit and using $|\varphi_\e'|\leq \frac C\e$ we obtain 
	\[\int_0^1 r^{M-1}\psi'\varphi_\e' \, dr\to \bar{r}^{M-1}\psi'(\bar r)-t^{M-1}\psi'(t) .\]
	Putting together these estimates we obtain for almost every $\bar r \in (0,t)$ that 
	\begin{equation}\label{passaggio}
	\bar{r}^{M-1}\psi'(\bar r)= t^{M-1}\psi'(t)+\int_{\bar r}^t r^{M-1}a(r)\psi\, dr+ \widehat \nu \int_{\bar r}^t r^{M-3}\psi\, dr .\end{equation}
	Similarly \eqref{passaggio} holds also for almost every $\bar r \in (t,1)$: it suffices to take $\varphi_{\e}:=\varphi_{\e, t, \bar r}$ and argue in the same way. Eventually $\psi\in C^1(0,1]$, because the r.A.s.~of \eqref{passaggio} is continuous, and  \eqref{stima-ggn} follows since 
	\[\int_{\bar r}^1 r^{M-3}\psi\, dr\leq  \left( \int_{\bar r}^1 r^{M-3}\, dr\right)^{\frac 12}\left( \int_0^1 r^{M-3}\psi^2\, dr\right)^{\frac 12}\leq \left\{
	\begin{array} {ll} 
	C & \text{ when }M>2 ,\\
	C\left(1+\log \bar r\right) & \text{ when }M=2 .
	\end{array}\right.
	\]
	To complete the proof of the first statement in \eqref{decay-psi} we let $\varphi_\e:=\varphi_{\e,r_n,R}$ for some $r_n,R\in(0,1)$ and use $\xi_\e:=r^\theta\varphi_\e$ as test function in \eqref{weak-radial-general}. Then, letting $\e\to 0$ and reasoning as before we obtain estimate (2.23) in \cite{GGN} or (A.32) of \cite{DGG} depending on $M$. The rest of the proof follows as in the papers \cite{GGN} and \cite{DGG} and we omit it.
	Estimates \eqref{decay-psi} and \eqref{passaggio} then imply that $\psi\in C[0,1]\cap C^1(0,1]$ and prove the regularity of $\psi$.
	
	Finally, once we have that $\psi=O(r^\theta)$ as $r\to 0$ with $\theta>0$, we also have that $r^{M-3}\psi\in L^1(0,1)$. We can then repeat the previous argument testing the weak equation for $\psi$ with $\varphi_{\e,a,b}$ with $a=0$ and $b=\bar r$, getting that
		\[
		\bar{r}^{M-1}\psi'(\bar r)= -\int_0^{\bar r} r^{M-1}a(r)\psi\, dr- \widehat \nu \int_0^{\bar r} r^{M-3}\psi\, dr\]
		Inserting the estimate  $\psi=O(r^\theta)$ in this last identity we obtain also the second estimate in \eqref{decay-psi} concluding the proof.}
\end{proof}

\begin{remark}\label{radial-general-0} 
	If $M> 2$  and	$\psi\in {\mathcal H}_{0,M}=  H^1_{0,M}$ is  a weak solution to  \eqref{radial-singular-problem-M} with $\widehat{\nu} = 0$, then actually $\psi$ solves \eqref{radial-eigenvalue-problem-M} with ${\nu} = 0$ and Proposition \ref{radial-regular-0} provides an estimate similar to \eqref{decay-psi}, i.e
	\begin{equation}\label{decay-psi-0}\psi(r)=O(1)  \ \ \text{ and  } \ \ \psi'(r)= o(1) 	\ \ \text{ as }r\to 0 .\end{equation}
\end{remark}
}

As before, one can deal with the minimization problem 
\[
\widehat{\nu}_1:=\inf_{\substack{w\in\mathcal{H}_{0,M}\ w\neq 0}}\frac{\mathcal{Q}_{a,M}(w)}{\int_0^1 r^{M-3}w^2(x)\, dr},  
\]
which can be iterated as
\begin{equation}\label{radial-singular-M}
  \widehat{\nu}_{i}:=\inf_{\substack{w\in\mathcal{H}_{0,M}\ w\neq 0\\ w\underline \perp_{M}\{\psi_1,\dots,\psi_{i-1}\}}}\frac{\mathcal{Q}_{a,M}(w)}{\int_0^1 r^{M-3}w^2(x)\, dr}.  
\end{equation}
Here the quadratic form $\mathcal{Q}_{a,M}$ is as in \eqref{forma-quadratica-a-rad}. 
Taking advantage from the Hardy inequality in the space $H^1_{M,0}$ (see Proposition \ref{hardy}) and repeating the arguments of Propositions \ref{p-lambda-1}, \ref{prop-autov-successivi}  and Lemma \ref{lem-var} one can see that
\begin{proposition}\label{c-i-attained}
	When the infimum in \eqref{radial-singular-M} is $\widehat{\nu}_i<\left(\frac{M-2}{2}\right)^2$, then $\widehat{\nu}_i$ is attained at a function $\psi_i\in \mathcal H_{0,M}$ which satisfies \eqref{weak-radial-general} and therefore is a weak solution to \eqref{radial-singular-problem-M}.
Viceversa the eigenvalues of problem  \eqref{radial-singular-problem-M}
	which are less than $\left(\frac{M-2}{2}\right)^2$ 
	coincide with the numbers  $\widehat \nu_i$  defined in \eqref{radial-singular-M}.
\end{proposition}

 In force of the variational characterization in Proposition \ref{c-i-attained} and of the regularity of eigenfunctions pointed out in Proposition \ref{prop-prel-3}, one can repeat the argument used in Subsection \ref{sse:prel-radial} to show that the main properties of the eigenvalues and of the related eigenfunctions still hold until $\widehat{\nu}_i<\left(\frac{M-2}{2}\right)^2$, in particular: 

\

\noindent {\bf{Property 1.}} {\em{The first eigenvalue $\widehat{\nu}_1$ is simple and the first eigenfunction $\psi_1$ is strictly positive (or negative) in $(0,1)$.}}

\medskip

\noindent {\bf{Property 2.}} {\em{Eigenfunctions related to different eigenvalues $\widehat{\nu}_i$ are orthogonal in ${\mathcal L}_M$.}}

\medskip

\noindent {\bf{Property 3.}} {\em{The $i$-th eigenfunction $\psi_i$  has at most $i$ nodal domains.}}

\medskip

Also the analogous of Propositions \ref{prop-prel-1} and \ref{prop-prel-0} continue to hold:

\begin{proposition}\label{prop-prel-2}
	Let $\nu_i$ be the generalized radial eigenvalues  defined in \eqref{Rayleigh-rad-M} and $\widehat{\nu}_i$ be the generalized singular  eigenvalues as defined in \eqref{radial-singular-M}. Then the number of negative eigenvalues $\nu_i$ coincides with the number of negative eigenvalues $ \widehat \nu_{i}$ corresponding to the same function $a(r)$.
	Moreover if $M>2$, $\nu_i=0$ if and only if $ \widehat \nu_{i}=0$ for the same index $i$.
\end{proposition}

The proof of these properties follows exactly as in the previous case and so we omit it.
\begin{remark}Proposition \ref{prop-prel-2} and Remark \ref{faggio}
    together with the variational characterization of the standard eigenvalues $\nu_i$ in \eqref{Rayleigh-rad-M} imply that the number of negative singular eigenvalues $\widehat \nu_i$ coincides with the maximal dimension of a subspace of $H^1_{0,M}$ in which the quadratic form $\mathcal Q_{a,M}$ defined in \eqref{forma-quadratica-a-rad} is negative defined. \end{remark}

Some more care is needed to show that {\bf{Property 5}} holds also for the generalized radial singular eigenvalues and eigenfunctions, namely 

\

\noindent {\bf{Property 5.}} {\em{Each generalized radial singular eigenvalue $\widehat\nu_i$ is simple and any $i$-th eigenfunction has exactly $i$ nodal domains. In particular  each radial singular eigenvalue $\widehat\L_i^{\rad}$ is simple and any $i$-th eigenfunction $\psi_i^{\rad}$ has exactly $i$ nodal domains.}}

\

The proof of this property relies, as in the previous case, on a Picone identity which holds also in this contest, but only for $r\in(0,1]$. It can be established repeating the proof of Lemma \ref{integrali-in-croce}, and can be stated as follows:
	
	\begin{lemma}[Picone identity]\label{integrali-in-croce-sing}
		Let $\psi_i, \psi_j$ weakly solve \eqref{radial-singular-problem-M} corresponding to $\widehat{\nu}_i , \widehat{\nu}_j$ respectively. Then \begin{align}\label{picone-identity-0-sing}
		\left(r^{M-1}\left(\psi_i'\psi_j- \psi_i\psi_j'\right) \right)'& = r^{M-3}(\widehat\nu_j-\widehat\nu_i) \psi_i \psi_j
		\intertext{for all $r\in(0,1]$.	If, in addition, $r$ is such that $\psi_j(r)\neq 0$, then at that $r$} \label{picone-identity-sing}
		\left(r^{M-1}\left(\psi_i'\psi_j- \psi_i\psi_j'\right)\frac{\psi_i}{\psi_j} \right)' &= 
		r^{M-3}(\widehat\nu_j-\widehat\nu_i) \psi_i^2 + r^{M-1} \left(\psi_i' - \psi'_j \frac{\psi_i}{\psi_j}\right)^2  .
		\end{align} 
		\end{lemma}

                We next show which changes are needed to prove Property 5 in the singular framework.
	
\begin{proof}[Proof of Property 5]
	The same arguments used in Subsection \ref{sse:prel-radial} to check  Property 5 for the generalized radial eigenvalue $\nu_i$, together with \eqref{picone-identity-0-sing} yield that  each eigenvalue $\widehat{\nu}_i$ is simple.
Next we show that $\psi_i$ has exactly $i$ nodal domains, by proving that it has $i-1$ zeros in $(0,1)$. As noticed before $\psi_1>0$ on $(0,1)$ provided that $ \widehat{\nu}_1<\left(\frac{M-2}{2}\right)^2$. If also $\widehat{\nu}_2 <\left(\frac{M-2}{2}\right)^2$  its eigenfunction $\psi_2$ is orthogonal to $\psi_1$ in the sense \eqref{scalar-HM} and so, using also Property 3, it has exactly $1$ zero in $(0,1)$. 
\\
Assume now that  $\psi_i$ has $i-1$ zeros in $(0,1)$, we want to show that $\psi_{i+1}$ has $i$ zeros in $(0,1)$ concluding the proof.
If $a,b\in(0,1]$ are two consecutive zeros of $\psi_i$ we can argue as in Subsection \ref{sse:prel-radial} (but using Lemma \ref{integrali-in-croce-sing} instead of Lemma \ref{integrali-in-croce} )  getting that $\psi_{i+1}$ has at least one zero in $(a,b)$. This means, letting $s\in(0,1)$ be the smallest positive zero of $\psi_i$,  that $\psi_{i+1}$  has at least $i-1$ zeros inside $(s,1)$. It remains to show that $\psi_{i+1}$ has another zero in $(0,s)$.	
	To this end we denote by $\sigma$ the first zero of $\psi_{i+1}$ in $(0,1)$ and assume by contradiction that $\sigma\ge s$. 
Let $w$ be the function which coincides with $\psi_i$ in $[0,s]$ and is null elsewhere; using it as a test function  in \eqref{weak-radial-general} we see that
\begin{align*}\int_0^{\sigma} r^{M-1} \left( |w'|^2 -  a(r) w^2 \right) dr
= \int_0^{s} r^{M-1} \left( |w'|^2 -  a(r) w^2 \right) dr \\
= \widehat\nu_i \int_0^{s} r^{M-3}w^2 dr = \widehat\nu_{i} \int_0^{\sigma} r^{M-3} w^2 dr .
\end{align*}
Therefore
\[
\min_{\substack{w\in  \mathcal H_{0,M}(0,\sigma)\\ w\neq 0}} \frac{\int_0^{\sigma} r^{M-1}\left(|w'|^2-a(r)w^2\right) dr}{\int_0^{\sigma} r^{M-3} w^2(r)\, dr}   \le \widehat\nu_i ,
\]
which means that the first eigenvalue of  \eqref{radial-singular-problem-M} settled in $(0,\sigma)$ instead of $(0,1)$ is less or equal than $\widehat\nu_i$, and therefore strictly less that $\widehat\nu_{i+1}$.
On the other hand $\psi_{i+1}$ is an eigenfunction related to the eigenvalue $\widehat\nu_{i+1}$ for this restricted problem, and since it has fixed sign on $(0,\sigma)$ it follows by  Properties 1 and 2 that $\widehat\nu_{i+1}$ actually is the first eigenvalue. The contradiction implies that $\psi_{i+1}$ admits another zero in $(0,s)$ showing that it has $i$ zeros in $(0,1)$ and concluding the proof.
\end{proof}

\section{Decomposition of the singular eigenvalues and applications to semilinear problems}\label{sse:2-bis}

Eventually we  explain the motivations which lead to study the singular eigenvalue problem \eqref{singular-eigenvalue-problem} and \eqref{radial-singular-problem-M} instead of \eqref{eigenvalue-problem} and  \eqref{radial-eigenvalue-problem-M}. 

Let us recall that the Spherical Harmonics, that we denote by $Y_j$, are the eigenfunctions of the Laplace-Beltrami operator on the sphere $S^{N-1}$. Of course the operator $\left(-\Delta_{S^{N-1}}\right)^{-1}$ is positive compact and selfadjoint in $L^2(S^{N-1})$ and  as before it admits a sequence of eigenvalues $0=\l_1<\l_2\leq \l_j$ and eigenfunctions $Y_j(\theta)$ (where $\theta$ is the system of coordinates on the sphere induced by the spherical coordinates in the space $\R^N$) which form an Hilbert basis for $L^2(S^{N-1})$. 
Namely they satisfy
\begin{equation}\label{eigen-lapl-beltr}
-\Delta_{S^{N-1}}Y_j(\theta)=\l_j Y_j(\theta)\  \text{ for }\theta\in S^{N-1}
\end{equation}
the eigenvalues $\l_j$ are given by the well known values 
\begin{equation}\label{eigen-beltrami}
\l_j:=j(N+j-2)\ \ \text{ for }j=0,1,\dots
\end{equation} 
each of which has multiplicity 
\begin{equation}\label{multiplicity-beltrami}
N_j:=\begin{cases}
1 & \text{ when }j=0,\\
\frac{(N+2j-2)(N+j-3)!}{(N-2)!j!} & \text{ when }j\geq 1 .
\end{cases}
\end{equation}
These eigenfunctions $Y_j(\theta)$ are bounded in $L^{\infty}(S^{N-1})$ by standard regularity theory.
Observe that $N_0=1$ with corresponding eigenfunction $Y_0(\theta)=c$ with $c$ constant, and in dimension $2$,  $N_j=2$ for any $j\geq 1$ with corresponding eigenfunctions $Y_j(\theta)=a_j\cos j \theta+b_j\sin j\theta$.

When $a$ is a radial functions, the singular eigenvalues can be decomposed in radial and angular part as follows. 
\begin{proposition}\label{prop-prel-4}
Let $a$ be any radial function in $L^{\infty}(B)$.
The singular eigenvalues $\widehat \L_i<\left(\frac{N-2}2\right)^2$ defined  in \eqref{i+1-singular} can be decomposed in radial and angular part as 
	\begin{equation}\label{decomposition}
	\widehat \L _i=\widehat \L_k^{\rad}+\l_j \quad \text{ for some  $k\geq 1$ and $j\ge 0$,}
	\end{equation}
	and the functions
	\begin{equation}\label{decomp-autof}
	 \psi_i(x)= \psi_k^{\rad}(r)Y_j(\theta)
	\end{equation}
	are solutions to \eqref{singular-eigenvalue-problem} corresponding to $\widehat \L_i$. Conversely, if a singular radial eigenvalue $\widehat \L_k^{\rad}$ according to \eqref{radial-singular} is such that  $ \widehat \L_k^{\rad}< \left(\frac {N-2}2\right)^2 -\l_j$ for some $j\ge 0$, then $	\widehat \L _i$ given by \eqref{decomposition} is a singular eigenvalue for \eqref{i+1-singular} and the function defined by \eqref{decomp-autof} is a related eigenfunction.
	\end{proposition}

\begin{proof}
Assume that $\psi\in \mathcal{H}_0$ is an eigenfunction related to a singular eigenvalue $\widehat\L<\left(\frac{N-2}2\right)^2$. Since $\mathcal{H}_0\subseteq H^1_0(B)$ we can decompose $ \psi$ along spherical harmonics $Y_j(\theta)$ namely
\[ \psi(r,\theta)=\sum_{j=0}^{\infty} \psi_j(r)Y_j(\theta) \ \ \text{ for }r\in (0,1)\ , \, \theta\in S^{N-1}\]
where 
\begin{equation}\label{psi-j}
\psi_j(r):=\int_{S^{N-1}}  \psi(r,\theta)Y_j(\theta)\, d\sigma(\theta)
\end{equation}
and if $\psi$ is non zero then at least one component $\psi_j(r)$ is non zero for some $j\geq 0$. 
Of course $\psi_j(1)=0$ for every $j$, moreover $\psi_j\in \mathcal H_{0,N}$ because
\[\begin{split}
&\int_0^1 r^{N-3}\psi_j^2\, dr=\int_0^1 r^{N-3}\left(\int_{S^{N-1}}\psi(r,\theta)Y_j(\theta)\, d\sigma(\theta)\right)^2\, dr\\ 
& \underset{\text{Jensen}}{\le} \int_0^1 r^{N-3}\int_{S^{N-1}}\left(\psi(r,\theta)Y_j(\theta)\right)^2\, d\sigma(\theta) \,dr\\
& \leq
\|Y_j\|_{\infty}^2\int_0^1\int_{S^{N-1}}r^{N-3} \psi^2 (r,\theta)\, d\sigma(\theta) \,dr=\|Y_j\|_{\infty}^2\int_B |x|^{-2}\psi^2(x)\, dx<\infty
\end{split}\]
since $\psi\in \mathcal{H}_0$ and similarly
\[\begin{split}
& \int_0^1 r^{N-1} \left(\psi_j'\right)^2\, dr =\int_0^1 r^{N-1}\left(\int_{S^{N-1}}\psi'(r,\theta)Y_j(\theta)\, d\sigma(\theta) \right)^2\, dr\\ 
& \underset{\text{Jensen}}{\leq} \int_0^1 r^{N-1}\int_{S^{N-1}}\left(\psi'(r,\theta)Y_j(\theta)\right)^2\, d\sigma(\theta)  \,dr\\
& \leq
\|Y_j\|_{\infty}^2\int_0^1\int_{S^{N-1}}r^{N-1} \left(\psi'(r,\theta)\right)^2\, d\sigma(\theta) \,dr = \|Y_j\|_{\infty}^2 \int_B |\nabla \psi|^2\, dx .
\end{split}\]
By \eqref{psi-j}, for every $\varphi\in \mathcal H_{0,N}$
  we have 
  \begin{align} \nonumber 
      \int\limits_0^1r^{N-1}\psi_j' \varphi' \ dr=\int\limits_0^1 \int\limits_{S^{N-1}} r^{N-1}  \frac {\partial \psi}{\partial r} Y_j(\theta)\varphi'dr d\sigma(\theta) =\int\limits_0^1 \int\limits_{S^{N-1}} r^{N-1} \frac {\partial \psi}{\partial r}\frac {\partial Y_j(\theta)\varphi  }{\partial r}\ dr d\sigma(\theta) 
    \intertext{and using that $ \psi$ solves \eqref{singular-eigenvalue-problem} and that $a=a(|x|)$ is radial then}
 \nonumber 
      =-\int\limits_0^1 \int\limits_{S^{N-1}} r^{N-3}\nabla _{\theta} \psi\cdot \nabla _{\theta}(Y_j(\theta)\varphi)\ dr d\sigma(\theta)
      +\int\limits _0^1 \int\limits_{S^{N-1}} r^{N-1}a(|x|) \psi Y_j(\theta)\varphi\ dr d\sigma(\theta)  
      \\ \nonumber
      +\widehat \Lambda \int\limits_0^1 \int\limits_{S^{N-1}} r^{N-3} \psi Y_j(\theta)\varphi\ dr d\sigma(\theta)
      \\ \nonumber
      =-\int\limits_0^1 r^{N-3}\ dr  \int\limits_{S^{N-1}}\nabla _{\theta}( \psi \varphi) \cdot \nabla _{\theta}Y_j(\theta)\ d\sigma(\theta)
      +\int\limits _0^1r^{N-1}a(r)\varphi  \psi_j  \ dr
       +\widehat \Lambda \int\limits_0^1 r^{N-3} \varphi \psi_j \ dr 
     \\\nonumber
       = -\l_j\int\limits_0^1  r^{N-3}\varphi\ dr\int\limits_{S^{N-1}} \psi Y_j(\theta)\ d\sigma(\theta) +\int\limits _0^1r^{N-1}a(r)\varphi  \psi_j  \ dr +\widehat \Lambda \int\limits_0^1 r^{N-3} \varphi \psi_j \ dr 
       \\  \label{mi-serve-dopo}
       =\int\limits _0^1r^{N-1}a(r)\varphi  \psi_j  \ dr+\big(\widehat \Lambda-\l_j\big)\int\limits_0^1 r^{N-3} \varphi \psi_j \ dr 
    \end{align}
  meaning that $\psi_j$ is a weak solution to \eqref{radial-singular-problem} corresponding to $\widehat \Lambda-\l_j$.

From \eqref{mi-serve-dopo} and the characterization of the radial singular eigenvalues in Proposition \ref{c-i-attained}  we have that the value $\widehat \L-\l_j$, which is strictly less than $\left(\frac {N-2}2\right)^2$, is a radial singular eigenvalue for $L_a$ as defined in \eqref{radial-singular}, namely \eqref{decomposition} holds for some $k\geq 1$.

The reverse implication holds as well, namely if $\widehat \L_k^{\rad}+\l_j<\left(\frac {N-2}2\right)^2 $ for some radial singular eigenvalue $\widehat \L_k^{\rad}$ with associated eigenfunction $ \psi_k^{\rad}\in \mathcal H_{0,\rad}$
and for one eigenvalue $\l_j$ of Laplace Beltrami, then the function $\Psi:= \psi_k^{\rad}(r)Y_j(\theta)$ belongs to $\mathcal{H}_0$. Indeed
\[\int_B\frac {\Psi^2}{|x|^2}\, dx= \int_0^1 r^{N-3} \left( \psi_k^{\rad}\right)^2\, dr \int_{S^{N-1}}Y_j^2(\theta)\, d\sigma(\theta) \leq C ,\]
\[\begin{split}
\int\limits_B|\nabla \Psi|^2\, dx\leq & C\int\limits_0^1 r^{N-1}\left(\left( \psi_k^{\rad}\right)'\right)^2\int\limits_{S^{N-1}}Y_j^2(\theta)\, d\sigma(\theta) \\
&+\int\limits_0^1 r^{N-3} \left( \psi_k^{\rad}\right)^2\, dr\int\limits_{S^{N-1}}|\nabla_{\theta}Y_j|^2\, d\sigma(\theta) \leq C.\end{split}\]
Moreover $\Psi$ weakly solves \eqref{singular-eigenvalue-problem} corresponding to $\widehat \L=\widehat \L_k^{\rad}+\l_j<0$. Indeed let $\varphi\in \mathcal H_0$ then
\begin{align*}
\int\limits_B \nabla \Psi\nabla \varphi\, dx =\int\limits_B \frac{\partial \Psi}{\partial r}  \frac{\de \varphi}{\de r}+\frac 1{r^2}\nabla_{\theta}\Psi\nabla_{\theta}\varphi\, dx\\
=\int\limits_{S^{N-1}}Y_j(\theta)\, d\sigma(\theta) \int\limits_0^1r^{N-1}\left( \psi_k^{\rad}\right)'\frac{\partial \varphi}{\partial r}\, dr
+\int\limits_0^1r^{N-3} \psi_k^{\rad}\, dr\int\limits_{S^{N-1}}\nabla_{\theta}Y_j\nabla_{\theta}\varphi\, d\sigma(\theta) \\
=\int\limits_{S^{N-1}}Y_j(\theta)\, d\sigma(\theta) \int\limits_0^1r^{N-1}\left(a(r)+\frac{\widehat \L_k^{\rad}}{r^2}\right) \psi_k^{\rad}\varphi\, dr
+\int\limits_0^1r^{N-3} \psi_k^{\rad}\, dr\l_j\int\limits_{S^{N-1}}Y_j(\theta)\varphi\, d\sigma(\theta)\\
=\int\limits_Ba(|x|)\Psi\varphi+\frac {\widehat \L_k^{\rad}+\l_j}{|x|^2}\Psi\varphi\, dx
\end{align*}
\end{proof}

When $N\ge 3$ the decomposition in \eqref{decomposition} holds as well in the case that $ L_a$ admits $0$ as an eigenvalue. The case $N=2$ is more delicate because in that case a regular eigenfunction does not necessarily belong to ${\mathcal{H}}_{0}$. Nevertheless a similar decomposition continues to hold, indeed we have the following:

\begin{proposition}\label{prop-prel-5}
	Let $a$ be any radial function in $L^{\infty}(B)$. 
When  $N=2$ the equation $L_a=0$ admits a  solution if and only if either 
\begin{equation}\label{decom-zero}
\widehat \L_k^\rad=-\l_j
\end{equation}
for some $j,k\geq 1$ and the corresponding solutions have the expression in \eqref{decomp-autof},
or $L_a=0$ has a radial solution in $H^1_{0,N}$ (i.e. $0$ is a radial eigenvalue as defined in \eqref{Rayleigh-rad}).
\end{proposition}
\begin{proof}
Let $w\in H^1_0(B)$ solves \eqref{eigenvalue-problem} with $\Lambda =0$. By Remark \ref{regolarita} then $w\in  C^{1,\beta}(\bar B)$.
Projecting $w$ along the spherical harmonics $Y_j=A_j\cos j\theta +B_j \sin j\theta$, $j\geq 0$  for $\theta\in [0,2\pi]$ and for suitable constants $A_j,B_j\in \R$ gives a sequence   
\begin{equation}\label{sequence-w-j}
w_j(r):=\int_0^{2\pi} w(r\cos \theta,r\sin\theta) \big(A_j\cos j\theta +B_j \sin j\theta\big) \ d\theta\end{equation}
defined for $r\in(0,1]$ and, by the regularity of $w$
\[\lim_{r\to 0}w_j(r)=w(0,0)\int_0^{2\pi} \big(A_j\cos j\theta +B_j \sin j\theta\big) \ d\theta\]
so that $\lim_{r\to 0}w_j(r)=0$ for  $j\geq 1$, while $\lim_{r\to 0}w_0(r)=A_0w(0,0)$ meaning that $w_0$ does not belong to $\mathcal H_{0,N}$, in general. \\
In any case by the regularity of $w$ for any $\varphi\in H^1_{0,N}$  we have that
\[\begin{split}
& \int_0^1 r^{N-1}w_0'\varphi' \ dr=\int_0^1 r^{N-1}\varphi' \int_0^{2\pi} \frac{\partial}{\partial r}w(r\cos \theta,r\sin\theta)\ d\theta\ dr
\intertext{and  using that $w$ solves \eqref{eigenvalue-problem} with $\Lambda =0$ gives}
&=\int_0^1\int_0^{2\pi}r^{N-1}\Big(-\frac 1{r^2}\frac{\partial w}{\partial \theta} \frac{\partial \varphi}{\partial \theta} +a(r)w \varphi\Big)\ d\theta\ dr=\int_0^1r^{N-1}a(r) w_0(r)\varphi \ dr
\end{split}\]
meaning that $w_0$ is a weak solution to \eqref{radial-eigenvalue-problem} corresponding an eigenvalue $\L_i^\rad=0$ for some $i\geq 1$. 

When $j\geq 1$ instead by \eqref{sequence-w-j} and the regularity of $w$ we have 
\[\begin{split}
& \frac{\partial}{\partial r}w_j(r)=\frac{\partial}{\partial r}\int_0^{2\pi}w(r\cos \theta,r\sin\theta) \big(A_j\cos j\theta +B_j \sin j\theta\big) \ d\theta\\
&=\int_0^{2\pi}\frac{\partial}{\partial r}w(r\cos \theta,r\sin\theta) \big(A_j\cos j\theta +B_j \sin j\theta\big) \ d\theta\\
&=\int_0^{2\pi}\Big[ \frac{\partial w}{\partial x_1}\cos \theta+  \frac{\partial w}{\partial x_2}\sin \theta\Big]\big(A_j\cos j\theta +B_j \sin j\theta\big) \ d\theta
\end{split}\]
so that
\[\lim_{r\to 0}\frac{\partial}{\partial r}w_j(r)=\int_0^{2\pi}\Big[ \frac{\partial w}{\partial x_1}(0,0)\cos \theta+  \frac{\partial w}{\partial x_2}(0,0)\sin \theta\Big]\big(A_j\cos j\theta +B_j \sin j\theta\big) \ d\theta\]
which implies
\[\big|\frac{\partial w_j}{\partial r}(r)\big|\leq C \ \ \ \text{ and}\ \ \ |w_j(r)|\leq Cr.\]
Then $w_j\in \mathcal H_{0,N}$ and then, reasoning as in \eqref{mi-serve-dopo} gives that $w_j$ is a weak solution to \eqref{radial-singular-problem-M} corresponding to $-\l_j$.  This proves \eqref{decom-zero}.
\end{proof}

Eventually  we are in the position to address the semilinear problem
\[\tag{\ref{general-f}}
\left\{\begin{array}{ll}
-\Delta u = f(|x|,u) \qquad & \text{ in } B, \\
u= 0 & \text{ on } \partial B.
\end{array} \right.
\]
As explained in the introduction we assume that $f$ satisfies A.1 and 
we consider weak solutions $u$ that satisfies A.2, namely such that $f_u(|x|,u(x))\in L^{\infty}(B)$, so that the linearized operator
\begin{align}
\label{linearized}
L_u(w)&:=-\Delta w-f_u(|x|,u)w
\intertext{ and the associated quadratic form }
\label{forma-quadratica}
{\mathcal Q}_u(w)& :=\int_B\left(|\nabla w|^2 -f_u(|x|,u)\,w^2\right) dx
\end{align}
fall within the analysis performed in  the previous sections.
In the remaining of this section  $\Lambda_i$, $\L^{\rad}_i$, $\widehat\L_i$, and $\widehat{\L}^{\rad}_i$ stand for the eigenvalues defined respectively in \eqref{Rayleigh}, \eqref{Rayleigh-rad},  \eqref{i+1-singular}, and \eqref{radial-singular}, with $a=f_u(|x|,u(x))$.	
Recall that a weak solution $u$ is said degenerate if the linearized equation $L_u w=0$ admits a nontrivial weak solution $w\in H^1_0(B)$, equivalently if
$\L_i=0$ for some index $i$. The Morse index of $u$, that we denote hereafter by $m(u)$, is instead the maximal dimension of a subspace of $H^1_0(B)$ in which the quadratic form $\mathcal Q_u$ is negative defined, or equivalently, since $L_u$ is compact, is the number, counted with multiplicity, of
$\L_i<0$. As $u$ is radial, we can consider the linearized operator $L_u$ and the quadratic form $\mathcal Q_u$ restricted to some symmetric space $H^1_{0,\mathcal G}(B)$, see \eqref{H-1-G}, where $\mathcal G$ stands for a subgroup of the orthogonal group $O(N)$. 
Then we can say that $u$ is $\mathcal{G}$-degenerate if $\L_i^{\mathcal G}=0$ for some index $i$
and the $\mathcal{G}$-Morse index of $u$ is the number, counted with multiplicity, of the negative eigenvalues $\L_i^{\mathcal G}$.
Finally when $\mathcal G =O(N)$ we say that $u$ is radially degenerate if $\L_i^{\rad}=0$ for some index $i$
and the radial Morse index is number of the negative radial eigenvalues $\L_i^{\rad}$.

\

Since we are in the framework of Section \ref{se:singular}, Propositions \ref{prop-prel-1} and \ref{prop-autov-rad} directly yield 
Proposition \ref{prop-1-2bis}. Furthermore  Theorem \ref{general-morse-formula} can be proved starting from  Proposition \ref{prop-prel-1}. 

\begin{proof}[Proof of  Theorem \ref{general-morse-formula}]
	By Proposition \ref{prop-prel-1} the Morse index of $u$ can be computed by summing the multiplicity of the negative singular eigenvalues $\widehat\L_i$ defined in \eqref{i+1-singular} with $a=f_u(|x|,u(x))$. Besides when $u$ is radial Proposition \ref{prop-prel-4} applies so that $\widehat\L_i <0$ iff $\widehat\L_i=\widehat\L^{\rad}_k+\lambda_j$ for some radial singular eigenvalue $\widehat\L_k^{\rad}<0$  and the related eigenfunctions are described by formula \eqref{decomp-autofunz-f(u)}. 
 Since  all the radial singular eigenvalues $\widehat\L_k^{\rad}$ are simple by the Property 5 recalled in Subsection \ref{sse:general}, the multiplicity of $\widehat\L_i$ as an eigenvalue of \eqref{i+1-singular} is obtained by summing an amount $N_j$ (the multiplicity of $\l_j$ as an eigenvalues of the Laplace-Beltrami operator recalled in \eqref{multiplicity-beltrami}) for every index $k$ such that $\widehat\L_k^{\rad}= \widehat\L_i -\l_j$.
	From the inverse viewpoint, the contribution of every  radial singular eigenvalues $\widehat\L_k^{\rad}$ to the Morse index of $u$ is obtained by summing an amount $N_j$ for every index $j$ such that $\widehat\L^{\rad}_k+\lambda_j<0$.
	Formula \eqref{tag-2} readily follows remembering \eqref{eigen-beltrami}.
\end{proof}

Eventually Proposition \ref{non-degeneracy} follows by Propositions \ref{prop-prel-0}, \ref{prop-autov-rad}, via the decomposition in Proposition \ref{prop-prel-4} (for $N\ge 3$) and \ref{prop-prel-5} (for $N=2$).

As a corollary of the decomposition of the singular eigenvalues we obtain also a formula to compute the Morse index and characterize the degeneracy in symmetric spaces. To this end we let $\mathcal G$ be any subgroup of the orthogonal group $O(N)$.
\begin{corollary}\label{cor-morse-sim}
	Assume { A.1}, take $u$ be a  radial weak solution to \eqref{general-f} satisfying {A.2} and let $m_{\rad}$ its radial Morse index. Then the $\mathcal G$-Morse index of $u$ is given by
	\begin{equation}\label{tag-2-sym}\begin{split}
	m^{\mathcal{G}}(u)=& \sum\limits_{i=1}^{m_\rad}\sum\limits_{j=0}^{\lceil J_i -1\rceil } N_j^{\mathcal G} \qquad \quad \mbox{where} \\
	J_i= & \sqrt{\left(\frac{\n-2}{2}\right)^2-\widehat\L^{\rad}_i}-\frac{\n-2}{2}
	\end{split}\end{equation}
	and $N_j^{\mathcal G}$ stands for the multiplicity of the $j^{th}$ eigenvalue of the Laplace-Beltrami operator in $H^1_{0,\mathcal G}$.
	Moreover $u$ is $\mathcal G$-degenerate if and only if it is radially degenerate or
	\begin{equation}\label{non-radial-degeneracy-sym}
	\widehat \L_k^{\rad} =  -j (N-2+j)\qquad \mbox{for some $k, j \ge 1$}\end{equation}	
	for some $j$ such that $N_j^{\mathcal G}\neq 0$.
\end{corollary}
Here $N_j^{\mathcal G}$ stands for the number of the $j^{th}$-spherical harmonics which are $\mathcal G$-invariant and, depending on $\mathcal G$, can be zero for many values of $j$.

\subsection{H\'enon type problems}

Here we focus on the particular case of H\'enon type nonlinearities, namely
\begin{equation}\label{general-f-H-5}
\left\{\begin{array}{ll}
-\Delta u = |x|^{\alpha}f(u) \qquad & \text{ in } B, \\
u= 0 & \text{ on } \partial B,
\end{array} \right.\end{equation}
where $\a\ge 0$ is a real parameter, under the  general assumptions 
\begin{enumerate}[{H.}1]
	\item $f\in W^{1,1}_{\loc}(\R)$,
	\item  $u$ is a weak solution to \eqref{general-f-H-5} such that $a(x):=f'(u(x))\in L^{\infty}(B)$.
\end{enumerate}
Let us remark that, thanks to the equivalence between $H^1_{\rad}(B)$ and $H^1_N$ pointed out in \cite[Theorem 2.3]{DFetal},  $u(x)$ is a weak radial solution to \eqref{general-f-H-5} if and only if  the real function $u(r)=u(x)$ for $r=|x|$ belongs to $H^1_{0,N}$ and verify
	\begin{equation}\label{henon-radial-weak-sol}
	\int_0^1 r^{N-1}  u' \phi' dr =  \int_0^1 r^{N-1+ \a} f(u)\phi \, dr 
	\end{equation} 
	for every test function $\varphi \in H^1_{0,N}$.
Besides a classical radial solution is $u\in C[0,1]\cap C^2[0,1)$  which satisfies
\begin{equation}\label{general-f-radial}
\begin{cases} 
-u''-\frac{N-1}{r}u' =  r^\a f(u) & \qquad \text{ as } 0<r<1 , \\
u'(0)=0 , \quad u(1)=0. &
\end{cases}\end{equation}

When $\a=0$ \eqref{general-f-H-5} becomes the autonomous problem
\begin{equation}\label{general-f-auto-5}
\left\{\begin{array}{ll}
-\Delta u = f(u) \qquad & \text{ in } B, \\
u= 0 & \text{ on } \partial B.
\end{array} \right.\end{equation}
Indeed  the relation between \eqref{general-f-H-5}  and \eqref{general-f-auto-5} is deeper: 
in the radial setting they are linked by means of the transformation
\begin{equation}\label{transformation-henon}
t=r^{\frac{2+\a}{2}} ,\qquad w(t)=u(r) ,
\end{equation}
which has been introduced in \cite{GGN} and maps any  radial solution $u$ of 
\eqref{general-f-H-5}   into a  solution $w$ of  
\begin{equation}\label{lane-emden-radial-5}
\begin{cases}
- \left(t^{M-1} w^{\prime}\right)^{\prime}= \left(\frac{2}{2+\a}\right)^2 t^{M-1} f(w)  , \qquad  & 0<t< 1, \\
w'(0)=0, \quad w(1)=0
\end{cases}\end{equation}
where
\begin{align}\label{Malpha}
M & = M(N,\alpha):= \frac{2(N+\alpha)}{2+\alpha}\in[2,N]  .
\end{align} 
Again,  by  weak solution to \eqref{lane-emden-radial-5} we mean a function $w\in H^1_{0,M}$ such that
	\begin{equation}\label{lane-emden-radial-weak-sol}
	\int_0^1 t^{M-1}  w' \varphi' dt = \left(\frac{2}{2+\a}\right)^2  \int_0^1 t^{M-1} f(w)\varphi \, dt  
	\end{equation} 
	for any test function $\varphi\in H^1_{0,M}$, while clearly a classical solution is $w\in C[0,1]\cap C^2[0,1)$  which satisfies \eqref{lane-emden-radial-5}.
	The weak formulation \eqref{lane-emden-radial-weak-sol} is natural because when $M$ is an integer then $w$ satisfying \eqref{lane-emden-radial-weak-sol} is a  radial solution for an autonomous problem  of type \eqref{general-f-auto-5} in dimension $M$ with $\left(\frac{2}{2+\a}\right)^2 f$ instead of $f$. 
	In any case in the setting of assumptions H.1 and H.2 the change of variables \eqref{transformation-henon} creates a one-to one correspondence between radial solutions to \eqref{general-f-H-5}  and solutions to the ODE \eqref{lane-emden-radial-5}, both in classical and in weak sense.
	
	Let us better clarify these facts, together with the properties of the transformation \eqref{transformation-henon}.
	
	\begin{lemma}\label{H-M-N}
		Let $M$ and $N$ be linked by \eqref{Malpha}. The transformation \eqref{transformation-henon} is a bijection between   $H^1_{0,N}$ and  $H^1_{0,M}$, and also between $\mathcal H_{0,N}$ and $\mathcal H_{0,M}$.
	\end{lemma}
	\begin{proof}
	It is easy to see that $u(r)$ has a weak derivative $u'\in L^1_{\loc}(0,1)$ if and only if the same holds for $w(t)=u(r)$  with 
	\begin{equation}
	\label{ecco} u'(r)=\frac{2+\a}{2} r^{\frac{\a}{2}} w'\left(r^{\frac{2+\a}{2}}\right) \qquad \text{a.e.}
	\end{equation} 
	Actually for any smooth function $\phi\in C_0^{\infty}(0,1)$, performing the change of variables \eqref{transformation-henon} and writing $\varphi(t)=\phi(r)$ one sees that
	\begin{align*}\label{ecco } 
	-\int_0^1 u(r) \phi'(r) dr & = -\int_0^1 w(t) \varphi'(t) dt = \int_0^1 w'(t) \varphi(t) dt = \int_0^1 \frac{2+\a}{2} r^{\frac{\a}{2}} w'\left(r^{\frac{2+\a}{2}}\right) \phi(r) dr . \end{align*}
	In particular
	\begin{align*}
	\int_0^1 r^{N-1} \left(u'(r)\right)^2 dr & = \left(\frac{2+\a}{2}\right)^2\int_0^1 r^{N-1+\alpha} \left(w'(r^{\frac{2+\a}{2}})\right)^2 dr = \frac{2+\a}{2}\int_0^1 t^{M-1} \left(w'(t)\right)^2 dt 
	\end{align*}
	which shows that $u'\in L^2_{N}$ if and only if $w'\in L^2_{M}$.\\
	We go on and check that  $u\in H^1_{0,N}$ if and only if $w\in H^1_{0,M}$. If $u\in H^1_{0,N}$ then $w\in H^1_{0,M}$ because
	\begin{align*}
	\int_0^1 t^{M-1} w^2(t) dt  & = \frac{2+\a}{2}\int_0^1 r^{N-1+\alpha} u^2(r) dr \le \frac{2+\a}{2}\int_0^1 r^{N-1} u^2(r) dr <+\infty.
	\end{align*}
	On the other hand if $w\in H^1_{0,M}$ then $u$ has a weak derivative and fulfills \eqref{ffci} with $u(1)=0$, so that  Jensen's inequality gives 
	\begin{align*}
	\int_0^1 r^{N-1} u^2(r) dr & = \int_0^1 r^{N-1} \left(\int_r^1 u'(\rho) \, d\rho\right)^2 dr
	\le  \int_0^1 r^{N-1}(1-r) \int_r^1 \left(u'(\rho)\right)^2 d\rho \, dr \\
	& =\int_0^1\int_0^{\rho}r^{N-1}(1-r)  dr \left(u'(\rho)\right)^2 d\rho
	\le \frac{1}{N}\int_0^1 \rho^{N-1} \left(u'(\rho)\right)^2 dr \\
	& = \frac{2+\a}{2N}\int_0^1 t^{M-1} \left(w'(t)\right)^2 dt <+\infty .
	\end{align*}	
	Eventually $u\in \mathcal H_{0,N}$ if and only if $w\in \mathcal H_{0,M}$, since 
	\begin{align*}
	\int_0^1 t^{M-3} w^2(t) dt  & = \frac{2+\a}{2}\int_0^1 r^{N-3} u^2(r) dr .
	\end{align*}
\end{proof}

The equivalence between the two definitions of weak solutions, \eqref{henon-radial-weak-sol} and  \eqref{lane-emden-radial-weak-sol},  is a straightforward consequence of Lemma \ref{H-M-N}. 
\begin{proposition}\label{prop-4.1-weak} 
	If $u$ is a weak radial solution to \eqref{general-f-H-5}, then $w$ given by  \eqref{transformation-henon} is a weak solution to \eqref{lane-emden-radial-5} with $M$ as in \eqref{Malpha}. Conversely if  $w$ is a weak solution to \eqref{lane-emden-radial-5}, after choosing $\a\ge 0$ and $N\in \mathbb N$, $N\ge 2$ such that \eqref{Malpha} is satisfied, the function $u$ defined on the $N$-dimensional unit ball according to \eqref{transformation-henon} is a   weak radial solution to \eqref{general-f-H-5}.
\end{proposition}
\begin{proof}
Assume for instance that $w(t)$ satisfies  \eqref{lane-emden-radial-weak-sol} and take any $\phi\in 
	H^1_{0,N}$. Lemma \ref{H-M-N} assures that $u(r)=w(t)\in H^1_{0,N} $ and $\varphi(t)=\phi(r) \in H^1_{0,M}$ and both couples satisfy \eqref{ecco}. Therefore
	\begin{align*}
	\int_0^1 r^{N-1} u'(r) \phi'(r) dr  \underset{\text{\eqref{ecco}}}{=}  \left(\frac{2+\a}{2}\right)^2\int_0^1 r^{N-1+\a} w'(r^{\frac{2+\a}{2}}) \varphi'(r^{\frac{2+\a}{2}}) dr \\  \underset{\text{\eqref{transformation-henon}}}{=} \frac{2+\a}{2}\int_0^1 t^{M-1} w'(t) \varphi'(t) dt 
	\underset{\eqref{lane-emden-radial-weak-sol}}{=} 
	\frac{2}{2+\a}\int_0^1 t^{M-1} f\left(w(t)\right) \varphi(t) dt 
	\\
	\underset{\text{\eqref{transformation-henon}}}{=} \int_0^1 r^{N-1+\a} f\left(u(r)\right) \phi(r) dr ,
	\end{align*}
	which is \eqref{henon-radial-weak-sol}.
	The opposite implication holds similarly.
\end{proof}

The assumption H.2, that we have used so far to guarantee the compactness of the linearized operator $L_u$, actually has a strong regularizing effect.
\begin{proposition}\label{reg-zero} 	
Assume that $\a\geq 0$ and $f$ satisfies $H.1$ and take $u$ a radial weak solution to \eqref{general-f-H-5} satisfying $H.2$. Then $u\in C^2[0,1]$ is also a classical solution and satisfies $u'(0)=0$, and $u''(0)=0$ when $\a>0$, $u''= -\frac1N  f(u(0))$ when $\a=0$.
\end{proposition}
\begin{proof}
	Assumptions H.1 and H.2 imply that $f$ is continuous and that along a solution $u$ 
	\[|f(u)|=|f(0)+\int_0^u f'(t)\ dt|\leq | f(0)|+C|u|\leq C+C|u|.\]
	This inequality together with the same arguments used in the proof of Proposition 2.2  ensure that
	\begin{equation}\label{passo}
	u'(r)=-r^{1-N}\int_0^r s^{N-1+\a }f(u(s))\ ds
	\end{equation}
	which implies, since $u$ is continuous in $(0,1)$, (because $u\in H^1_{0,N}$)  that $u\in C^1(0,1)$. Next
	\[u(r)=u(1)-\int_r^1u'(t)\ dt=-\int_r^1t^{1-N}\int_0^t s^{N-1+\a } f(u(s))\ ds dt\]
	and then
	\[|u(r)|\leq C\int_r^1t^{1-N}\int_0^r s^{N-1+\a} (1+|u(s)|)\ ds dt\]
	which, together with the Radial Lemma (see \cite{Ni} or Lemma 5.2 here), proves that $u$ is bounded in $r=0$ and then $u\in C[0,1)$. We do not give the details since are exactly the same as in the proof of Proposition 2.2. This is enough to prove that $u'(r)$ has a limit as $r\to 0$, by \eqref{passo}, so that $u\in C^1[0,1)$ and $u'(0)=0$.
	\\
	Beside $r^{N-1}u'$ has a weak derivative $r^{N-1+\a}f(u(r))\in C(0,1)$ so that the same arguments of Remark 2.3  ensure 
	\begin{equation}\label{u-secondo}
	u''(r)=-\frac{N-1}r u'-r^\a f(u)  \ \text{ for  }r\in (0,1)\end{equation}
	which easily implies that $u\in C^2(0,1)$ since $f$ and $u$ are continuous. 
	Finally from \eqref{u-secondo} and by the l'Hopital theorem we have
	\begin{align*}
	&\lim_{r\to 0^+}u''(r)=\lim_{r\to 0^+}\frac{N-1}{r^N} \int_0^r s^{N-1+\a}f(u(s))\ ds-r^\a f(u)\\
	&\ \ =\lim_{r\to 0^+}\frac{N-1}{N}r^\a f(u)-r^\a f(u)=\left\{
	\begin{array}{ll}
	0 & \text{ when }\a>0\\
	-\frac1N  f(u(0))& \text{ when }\a=0
	\end{array}\right.
	\end{align*}
	which shows that $u$ admits a second derivative in $r=0$ and assures $u\in C^2[0,1)$. The regularity at $r=1$ follows in a similar way.
\end{proof}

\begin{remark}
	The proof of Proposition \ref{reg-zero} is valid for any nonlinearity $f(|x|, u)$ under assumptions A.1 and A.2.
\end{remark}

In the same way one can see that
\begin{corollary}\label{cor-reg-zero}
Assume that $f$ satisfies $H.1$ and take $w$ a weak solution to \eqref{lane-emden-radial-5} satisfying $H.2$.  Then $w\in C^2[0,1]$ is also a classical solution and satisfies $w'(0)=0$, and $w''= -\frac1N  f(u(0))$.
\end{corollary}

Consequently we get the following

\begin{proposition}\label{prop-4.1} 
	 Assume $\a\geq 0$ and let $u$ and $w$ be related as in \eqref{transformation-henon}, $M=M(N,\a)$ as in \eqref{Malpha}. 
	If $f$ satisfies H.1, then $u$ is a (weak or classical) radial solution to \eqref{general-f-H-5} satisfying H.2 if and only if $w$ is a (weak or classical) solution to  \eqref{lane-emden-radial-5} satisfying H.2.
\end{proposition}
\begin{proof}
 It follows by Propositions \ref{prop-4.1-weak}, \ref{reg-zero} and Corollary \ref{cor-reg-zero}.
\end{proof}

The transformation \eqref{transformation-henon} is useful also in computing the Morse index and examining the degeneracy of radial solutions to \eqref{general-f-H-5}. By Theorems \ref{general-morse-formula} and \ref{non-degeneracy} such issues are related to the radial singular eigenvalues $\widehat \L^{\rad}_k$ characterized by the Sturm-Liouville problem \eqref{radial-singular-problem} with $a(r)=r^{\a}f'(u)$ and, when $N=2$, also to the radial standard eigenvalues $\L_k^\rad$ characterized by \eqref{radial-eigenvalue-problem}. 
Performing the change of variable \eqref{transformation-henon} inside   \eqref{radial-singular-problem} brings to Sturm-Liouville problems of type \eqref{radial-singular-problem-M} with $M$ given by \eqref{Malpha} and $a(t)= \left(\frac{2}{2+\a}\right)^2 f'(w(t))$, i.e.
\begin{equation}\label{radial-general-H}
\left\{\begin{array}{ll}
- \left(t^{M-1} \phi'\right)'- \left(\frac{2}{2+\a}\right)^2 t^{M-1} f' (w(t)) \phi = t^{M-3} \widehat{\nu}  \phi & \text{ for } t\in(0,1)\\
\phi\in  \mathcal H_{0,M} . &
\end{array} \right.
\end{equation} 
It is easily seen that 

\begin{lemma}	\label{general-eigenvalue-H-LE}
 Assume that $\a\geq 0$ and $f$ satisfies $H.1$ and take $u$ a radial weak solution to \eqref{general-f-H-5} satisfying $H.2$. Then	$\widehat \L^{\rad}_i<\left(\frac{N-2}{2}\right)^2$ is a radial singular eigenvalue for the linearized operator $L_u$ if and only if 
	\begin{equation}
	\label{relazione-autov}
	\widehat{\nu}_i = \left(\frac{2}{2+\a}\right)^2  \widehat\L^{\rad}_{i} < \left(\frac{M-2}{2}\right)^2
	\end{equation}
	is an eigenvalue for \eqref{radial-general-H}.
	Further $\psi_i\in\mathcal H_{0,N}$ is an eigenfunction related to $\widehat \L^{\rad}_i$ if and only if $\psi_i(r)=\phi_i(t)$, where $\phi_i\in\mathcal{H}_{0,M}$ is an eigenfunction for problem \eqref{radial-general-H} related to  $\widehat{\nu}_i$.
	\end{lemma}
 \begin{proof}
Take  $\psi_i\in\mathcal H_{0,N}$ an eigenfunction related to $\widehat \L^{\rad}_i<\left(\frac{N-2}{2}\right)^2$ and   $\varphi\in {\mathcal H}_{0,M}$ any text function.  Performing the change of variables \eqref{transformation-henon} and  letting  $\phi_i(t)= \psi_i(r)$, $\tilde{\varphi}(r)= \varphi(t)$, 
 	then $\phi_i \in {\mathcal H}_{0,M}$, $\tilde{\varphi}\in \mathcal H_{0,N}$ by Lemma \ref{H-M-N} and 
 	\begin{align*}
	\left(\frac{2}{2+\a}\right)^2  \widehat\L^{\rad}_{i} \int_0^1 t^{M-3} \phi_i  \varphi dt
	=\frac{2}{2+\a}\widehat\L^{\rad}_{i} \int_0^1 r^{N-3} \psi_i \tilde\varphi dr \\
	= \frac{2}{2+\a}\int_0^1 r^{N-1} \left(\psi'_i \tilde\varphi' - r^{\a} f'(u) \psi_i\tilde\varphi\right) dr \\
	\underset{ \eqref{ecco}}{=} \frac{2}{2+\a}\int_0^1 r^{N-1} \left(\left(\frac{2+\a}{2}\right)^2 r^{\a}{\phi'_i}\left(r^{\frac{2+\a}{2}}\right)\varphi'\left(r^{\frac{2+\a}{2}}\right)- r^{\a} f'(u) {\phi_i\left(r^{\frac{2+\a}{2}}\right) \varphi\left(r^{\frac{2+\a}{2}}\right)}\right) dr \\
	= \int_0^1 t^{M-1} \left({\phi'_i}\varphi'-\left(\frac{2}{2+\a}\right)^2  f'(w) {\phi_i}\varphi\right) dt,
	\end{align*}
	which shows that ${\phi_i}$ solves \eqref{radial-general-H} with $\widehat{\nu}=\left(\frac{2}{2+\a}\right)^2 \widehat{\L}^{\rad}_i$. The opposite implication follows similarly.
\end{proof}

Thanks to Lemma \ref{general-eigenvalue-H-LE}  one can easily deduce from Proposition \ref{prop-prel-4}  the ad-hoc formula to compute the radial and general Morse index of H\'enon type problems stated in Proposition \ref{general-morse-formula-H}.

 \begin{proof}[Proof of Proposition \ref{general-morse-formula-H}]
Remembering the decomposition of the singular eigenvalues $\widehat \L_i$ in Proposition \ref{prop-prel-4}, 	Lemma \ref{general-eigenvalue-H-LE} implies that the singular eigenvalue problem associated to the linearized operator $L_u$, namely
	\[
	\left\{\begin{array}{ll}
	-\Delta  \psi_i-|x|^\a f'(u)\psi_i =\frac{\widehat\L_i}{|x|^2}\psi_i & \text{ in } B\setminus\{0\}\\
	\psi_i= 0 & \text{ on } \partial B,
	\end{array} \right.
	\]
	admits as singular eigenvalues 
	\[\widehat \L_i=\left(\frac{2+\a}2\right)^2 \widehat \nu_k+\l_j ,\]
	and the corresponding eigenfunctions are
	\[\psi_i(x)=\phi_k\big( r^{\frac {2+\a}2}\big) Y_j(\theta) ,\]
	where $\phi_k$ is the eigenfunction to \eqref{radial-general-H} associated with $\widehat \nu_k$.  
In that way we get the decomposition formula \eqref{decomp-autofunz-f(u)-H}, while Propositions \ref{prop-autov-rad} and \ref{prop-prel-1}  yield respectively the characterization of the radial and general Morse index. 
	\end{proof}

Similarly Propositions \ref{prop-prel-4} and   \ref{prop-prel-5} furnish the characterization  of degeneracy in Proposition \ref{non-degeneracy-H}.

\begin{proof}[Proof of Proposition \ref{non-degeneracy-H}]
	When $N\ge 3$ it is known by Propositions \ref{prop-prel-0} and \ref{prop-prel-4} that $u$ is radially degenerate if  $\widehat{\L}_k^{\rad}=\L^{\rad}_k=0$ for some index $k$, and degenerate if $\widehat \L_k^{\rad}=- j(N-2+j)$. So Lemma \ref{general-eigenvalue-H-LE}  immediately gives the statement.
	\\
	When $N=2$ instead,  by definition $u$ is radially degenerate if and only if $\L_k^{\rad}=0$ for some index $k$. In that case using Lemma \ref{H-M-N} it is easy to see that  	{ $\nu=0$ is an eigenvalue for
		\begin{equation*}
	\left\{\begin{array}{ll}
	- \left(t^{M-1} \phi'\right)'- \left(\frac{2}{2+\a}\right)^2 t^{M-1} f' (w(t)) \phi  = t^{M-1} {\nu} \, \phi & \text{ for } t\in(0,1)\\
	\phi'(0)=0 , \ \phi(1)=0  . &
	\end{array} \right.
	\end{equation*} }
 Moreover, recalling also Proposition  \ref{prop-prel-2}, $0$ must be the $k^{th}$ eigenvalue.
	For what concerns nonradial degeneracy, the claim follows  by Proposition \ref{prop-prel-5}  like in the case $N\ge 3$.	
\end{proof}

Furthermore the decomposition formula \eqref{decomp-autofunz-f(u)-H} gives information also on the symmetric Morse index since the symmetries of the eigenfunctions are exactly the symmetries of the Spherical Harmonics.  Thus Propositions \ref{prop-autov-rad}, \ref{prop-prel-4},  together with Lemma \ref{general-eigenvalue-H-LE},  yield also the following
\begin{corollary}\label{cor-morse-sim-H}
Assume that $\a\geq 0$ and $f$ satisfies $H.1$ and take $u$ a radial weak solution to \eqref{general-f-H-5} satisfying $H.2$. 
	If $m_{\rad}$ stands for its radial Morse index, then the $\mathcal G$-Morse index of $u$ is given by 
	\begin{equation}\label{tag-2-sym-H}\begin{split} 
	m^{\mathcal{G}}(u)=\sum\limits_{i=1}^{m_\rad}\sum\limits_{j=0}^{\lceil J_i -1\rceil } N_j^{\mathcal G}  \quad \mbox{where} \\
	{	J_i= \frac{2+\a}2\left( \sqrt{\left(\frac{M-2}{2}\right)^2-\widehat\nu_i}-\frac{M-2}{2}\right)}
	\end{split}\end{equation}
	and $N_j^{\mathcal G}$ stands for the multiplicity of the $j^{th}$ eigenvalue of the Laplace-Beltrami operator in $H^1_{0,\mathcal G}$.
	Moreover $u$ is $\mathcal G$-degenerate if and only if it is radially degenerate or
	\begin{equation*}\label{non-radial-degeneracy-sym-H}
	\widehat \nu_k =  -\Big(\frac 2{2+\a}\Big)^2j (N-2+j)\qquad \mbox{for some $k, j \ge 1$ such that $N_j^{\mathcal G}\neq 0$.}\end{equation*}	
\end{corollary}
Notice that $N_j^{\mathcal G}$ stands for the number of the $j^{th}$-spherical harmonics which are $\mathcal G$-invariant and, depending on $\mathcal G$, can be zero for many values of $j$.

\section{Appendix: some useful properties of the spaces $H^1_M$ and $H^1_{0,M}$}\label{appendix}

For any $M\in\R$, $M\ge 2$, and $q\ge 1$ we  denote by $L^q_M$ the set of measurable functions $v:(0,1)\to\R$ such that
			\[ \int_0^1 r^{M-1} |v|^q dr < +\infty.\]
			Clearly $L^2_M$ is a  Hilbert space endowed with the product 
			\[ \langle v,w\rangle_M = \int_0^1 r^{M-1} v \, w \, dr ,\]
			which yields the orthogonality condition
			\[ v \perp_M w \, \Longleftrightarrow \, \int_0^1 r^{M-1} v \, w \, dr = 0 .\]
			Next we denote by $H^1_{M}$ the subspace of $L^2_M$  made up by that functions $v$ which have weak first order derivative in $L^2_M$, so that the norm 
	\[ \|v\|_M = \left(\int_0^1 r^{M-1} \left( v^2+|v'|^2\right) dr\right)^{\frac{1}{2}} \]
	is bounded. 
	Further  by \cite[VIII.2]{Bbook}  any function in $v\in H^1_{M}$ is almost everywhere equal to a function $\tilde v\in C(0,1]$ which is differentiable almost everywhere with 
	\begin{equation}\label{ffci-app}
\tilde v(r_2)-\tilde v(r_1) = \int_{r_1}^{r_2} v'(r) dr .\end{equation}
Therefore we may assume w.l.g.~that any $v\in H^1_{M}$ is continuous in $(0,1]$ and satisfies \eqref{ffci-app}.
This allows to introduce the set
\[H^1_{0,M} = \left\{ v\in H^1_M \, : \, v(1)=0\right\} .\]
Moreover if $M=N$ is an integer then $H^1_N$ actually is equal to $H^1_{\rad}(B)$ by \cite[Theorem 2.2 (2)]{DFetal}, and therefore $H^1_{0,N}=H^1_{\rad,0}(B)$.
Most of the general properties of the Sobolev spaces $H^1_{\rad}(B)$ and $H^1_{0,\rad}(B)$ are valid also in $H^1_M$ and $H^1_{0,M}$. Let us recall the ones which turn useful in our computations.

\begin{lemma}[Poincar\'e inequality]
Let $M\ge 2$; for any $v\in H^1_{0,M}$ we have
	\[ \int_0^1 t^{M-1} v^2(t) dt \le \frac{1}{M-1} \int_0^1 t^{M-1} \left(v'(t)\right)^2 dt .\]
\end{lemma}
\begin{proof}
	By \eqref{ffci-app} we have
	\begin{align*}
	\int_0^1 t^{M-1} v^2(t) dt & = 	\int_0^1 t^{M-1} \left(\int_t^1 v'(s) ds\right)^2 dt \\
	& \le \int_0^1 t^{M-1}  (1-t) \int_t^1 \left(v'(s)\right)^2 ds \,dt \le \int_0^1 t^{M-2} \int_t^1 \left(v'(s)\right)^2 ds \,dt \\
	&= \int_0^1  \left(v'(s)\right)^2  \int_0^s t^{M-2} dt \,ds = \frac{1}{M-1} \int_0^1 s^{M-1}  \left(v'(s)\right)^2  ds.
	\end{align*}
\end{proof}
Therefore in the space $H^1_{0,M}$ one can use the equivalent norm in given by \[\|v\|_{0,M} = \left(\int_0^1 t^{M-1} |v'|^2 dt\right)^{\frac{1}{2}}.\]

\begin{lemma}[Radial Lemma]\label{radial-lemma}
	Let $M\ge 2$; for any $v\in H^1_{0,M}$ we have
	\begin{equation}\label{radial-lemma-estimate} |v(t)| \le \begin{cases}
	\displaystyle	\left(\int_0^1r^{M-1}\left(v'(r)\right)^2 dr\right)^{\frac{1}{2}} \frac{t^{-\frac{M-2}{2}}}{\sqrt{M-2}} \quad & \text{ if } M>2, \\
	\displaystyle	\left(\int_0^1r^{M-1}\left(v'(r)\right)^2 dr\right)^{\frac{1}{2}} \left|\log t\right|^{\frac{1}{2}} & \text{ if } M=2.
	\end{cases}\end{equation}
\end{lemma}

\begin{proof}
	By \eqref{ffci-app} we have
	\begin{align*}
	|v(t)| & \le \int_t^1 |v'(s)| ds  \le \left(\int_t^1s^{1-M} ds\right)^{\frac{1}{2}}
	\left(\int_0^1r^{M-1}\left(v'(r)\right)^2 dr\right)^{\frac{1}{2}}
	\intertext{which gives, for $M>2$,}
	& = \frac{1}{\sqrt{M-2}} \left(\frac{1}{t^{M-2}}-1\right)^{\frac{1}{2}}
	\left(\int_0^1r^{M-1}\left(v'(r)\right)^2 dr\right)^{\frac{1}{2}}
	\le \frac{1}{\sqrt{M-2}} \frac{\left(\int_0^1r^{M-1}\left(v'(r)\right)^2 dr\right)^{\frac{1}{2}}}{t^{\frac{M-2}{2}}} ,
	\intertext{ or, when $M=2$}
	& =	\left|\log t\right|^{\frac{1}{2}}
	\left(\int_0^1r^{M-1}\left(v'(r)\right)^2 dr\right)^{\frac{1}{2}} .
	\end{align*}	
\end{proof}

\begin{lemma}[Sobolev embedding]\label{talenti}
	Let $M>2$ and $2^{\ast}_M=\frac{2M}{M-2}$. The space $H^1_M(0,+\infty)$ is continuously embedded in $L^{2^{\ast}_M}_M(0,+\infty)$ and the best constant
	\[S_M=\min\limits_{v\in H^1_M(0,+\infty)} \frac{\int_0^{+\infty} t^{M-1}\left(v'(t)\right)^2 dt}{\left(\int_0^{+\infty}t^{M-1} v^{2^{\ast}_M}(t) dt\right)^{\frac{2}{2^{\ast}_M}}} \]
	is achieved by any  Aubin-Talenti's bubble $U(t)=\left(a+b t^2\right)^{-\frac{M-2}{2}}$.
\end{lemma}

The just stated Sobolev embedding has been established  by  Aubin \cite{Aubin} and Talenti \cite{talenti}.
If $M$ is an integer, the embedding of $H^1_M= H^1_{\rad}(B)$ into $L^q(B)$ is compact for every $\mathcal Q<2^{\ast}_M$ for $M>2$ and for any $\mathcal Q$ if $M=2$. The same arguments can be repeated for any $M$ to obtain

\begin{lemma}[Compact Sobolev embedding]\label{sobolev}
	Let $M> 2$ and $2^{\ast}_M=\frac{2M}{M-2}$. The space $H^1_{0,M}$ is continuously embedded in $L^{2^{\ast}_M}_M$ and compactly embedded in $L^{q}_{M}$ for any $\mathcal Q<2^{\ast}_M$.
	Otherwise if $M=2$ then  $H^1_{0,M}$ is compactly embedded in $L^{q}_{M}$ for any $\mathcal Q$.
\end{lemma}
\begin{proof}
	By Lemma \ref{talenti}	$H^1_{0,M}$ is continuously embedded in $L^{q}_M$ for any $\mathcal Q\le 2^{\ast}_M$ if $M>2$.
	Moreover if $M=2$ Lemma \ref{radial-lemma} implies that $H^1_{0,M}$ is continuously embedded in any $L^q_M$  since  		\[\int_0^1 t |v|^q dt \le \|v\|_M^q \int_0^1 t |\log t|^{\frac{q}{2}} dt < \infty .\]
		Besides it is easy to see that for any  $M\ge 2$ $H^1_{M}$ is compactly embedded in $C[\e,1]$ for any $\e\in(0,1)$. Indeed by \eqref{ffci-app}
if $r_1,r_2>\e$ we have
\begin{align*}
\left|v(r_1)-v(r_2)\right| \le \int_{r_1}^{r_2} |v'(t)| dt \underset{\text{Holder}}{\le} 
\left(\int_{r_1}^{r_2} t^{1-M} dt\right)^{\frac{1}{2}} \|v\|_M \le \e^{\frac {1-M}2} \|v\|_M \sqrt{|r_1-r_2|}
\end{align*}
and the claim follows by the Ascoli Theorem.
Now, let $v_n$ be a bounded sequence in $H^1_{M}$: up to a subsequence it converges  locally uniformly to a function $v$, and Fatou's Lemma ensures that $v\in L^{2^{\ast}_M}_M$, or $v\in L^q_M$ for any $\mathcal Q$ if $M=2$. It is left to check that $v_n\to v$ in $L^q_M$ for $\mathcal Q<{2^{\ast}_M}$, or for any $\mathcal Q$ if $M=2$. 
We take first $M>2$, and $\mathcal Q<{2^{\ast}_M}$. For any fixed $\e\in(0,1)$ we compute
\begin{align*}
\int_0^1 t^{M-1} |v_n-v|^q dt  = &\int_0^{\e} t^{M-1} |v_n-v|^q dt + \int_{\e}^1 t^{M-1} |v_n-v|^q dt 
\intertext{and using \eqref{radial-lemma-estimate} to control the first integral and the fact that $v_n\to v$ in $C[\e,1]$ for the second one gives}
\le & C \int_0^{\e} t^{M-1-\frac{q(M-2)}{2}} dt +  \frac{1-\e^M}{M} \left(\sup\limits_{[\e,1]}|v_n-v|\right)^q
\\ = &  C \e^{M-\frac{q(M-2)}{2}}  +  \frac{1-\e^M}{M} \left(\sup\limits_{[\e,1]}|v_n-v|\right)^q.
\end{align*}
	Therefore, choosing first $\e$ and then $n$, the quantity $\int_0^1 t^{M-1} |v_n-v|^q dt $ can be made arbitrarily small.
	In a similar way, if $M=2$ and $\mathcal Q>1$ we have
\begin{align*}
\int_0^1 t |v_n-v|^q dt  = &\int_0^{\e} t |v_n-v|^q dt + \int_{\e}^1 t |v_n-v|^q dt \\
\le & C \int_0^{\e} t |\log t|^{\frac{q}{2}} dt +  \frac{1-\e^2}{2} \left(\sup\limits_{[\e,1]}|v_n-v|\right)^q
\end{align*}	
and the conclusion follows.	
\end{proof}

\begin{lemma}[Hardy inequality]\label{hardy}
	If $M>2$, then
	\[ \left(\frac{M-2}{2}\right)^2 \int_0^1 t^{M-3} v^2 dt \le  \int_0^1 t^{M-1} |v'|^2 dt\]
	for any $v\in H^1_{0,M}$.
\end{lemma}
\begin{proof}
	
For any $v\in H^1_{0,M}$ the function $u(t)= t^{\frac{M-2}{2}}v(t)$ is continuous on $(0,1]$ with $u(1)=0$, differentiable a.e. and bounded by the Radial Lemma \ref{radial-lemma}, since $M>2$.
Hence 
	\begin{align*}
	\left(\frac{M-2}{2}\right)^2 \int_0^1 t^{M-3} v^2 dt - \int_0^1 t^{M-1} |v'|^2 dt  =
          - \int_0^1 t |u'|^2 dt +  \frac{M-2}{2}\int_0^1 (u^2)' dt\\
          \le - \int_0^1 t |u'|^2 dt -\frac{M-2}2 \liminf\limits_{t\to 0} u^2(t) \le 0 .
	\end{align*} 
\end{proof}


\begin{thebibliography}{99}


\bibitem{AP} {\sc A. Aftalion, F.  Pacella},  Qualitative properties of nodal solutions of semilinear elliptic equations in radially symmetric domains (2004) {\it Comptes Rendus Mathematique}, 339 (5), pp. 339-344. DOI: 10.1016/j.crma.2004.07.004



\bibitem{A}{\sc A.L. Amadori}, On the  asymptotically linear H\'enon problem, (2019) arXiv 

\bibitem{AG-henon}{\sc A.L. Amadori, F. Gladiali}, 
Bifurcation and symmetry breaking for the H\'enon equation
{\em Advances in Differential Equations}, 19 (7-8) (2014), 755-782, https://projecteuclid.org/euclid.ade/1399395725.

\bibitem{AG17}{\sc A.L.~Amadori, F.~Gladiali},  Nonradial sign changing solutions to Lane-Emden problem in an annulus {\em Nonlinear Analysis, Theory, Methods and Applications} 155 (2017), 294-305.

\bibitem{AG-sing-2} {\sc A.L.~Amadori, F.~Gladiali}, On a singular eigenvalue problem and its applications in computing the Morse index of solutions to semilinear PDE's - Part II, (2019) 	arXiv


\bibitem{AG18}{\sc A.L.~Amadori, F.~Gladiali}, Asymptotic profile and Morse index of nodal radial solutions to the H\'enon problem, (2018) 	arXiv:1810.11046

\bibitem{AG18-2}{\sc A.L.~Amadori, F.~Gladiali}, The H\'enon problem with large exponent in the disc, (2019) arXiv:1904.05907

\bibitem{AGG}{\sc A.L.~Amadori, F.~Gladiali, M. Grossi}, Nodal solutions for lane-Emden problems in almost-annular domains
(2018) {\em Differential and Integral Equations}, 31 (3-4), pp. 257-272. 


\bibitem{AM}{\sc A.~Ambrosetti, A.~Malchiodi}, { Nonlinear analysis and semilinear elliptic problems},
{Cambridge Studies in Advanced Mathematics}, {104},
	{Cambridge University Press, Cambridge}, {2007},
	DOI: 10.1017/CBO9780511618260

\bibitem{Aubin}{ \sc T. Aubin}, 
Probl\'emes isop\'erim\'etriques et espaces de Sobolev.
{\it J. Differential Geometry }11 (1976), no. 4, 573-598. 


 \bibitem{BCW} {\sc Bartsch, T., Chang, K.-C., Wang, Z.-Q.}, On the Morse indices of sign changing solutions of nonlinear elliptic problems (2000) {\it Mathematische Zeitschrift}, 233 (4), pp. 655-677. DOI: 10.1007/s002090050492

\bibitem{BSW}{\sc T. Bartsch, A. Szulkin, M. Willem}, Morse theory and nonlinear differential equations. Handbook of global analysis, 41-73, 1211, Elsevier Sci.B.V., Amsterdam, 2008.

\bibitem{Bartsch-Weth}{\sc T. Bartsch, T. Weth}, A note on additional properties of sign changing solutions to superlinear equations, {\em  Topol. Methods Nonlinear Anal.} 22 (2003), 1–14.

\bibitem{BW93} {\sc T. Bartsch, M. Willem},  Infinitely many radial solutions of a semilinear elliptic problem on ${\mathbb R}^N$(1993) {\em Archive for Rational Mechanics and Analysis}, 124 (3), pp. 261-276. DOI: 10.1007/BF00953069



\bibitem{Bbook} {\sc A. Brezis}, {\it Functional Analysis, Sobolev Spaces and Partial Differential Equations}, Universitext, Springer, New York, 2010.

  

\bibitem{CastroCossioNeuberger}{\sc A. Castro, J. Cossio, J.Neuberger},  A sign-changing solution for a superlinear Dirichlet problem. Rocky Mountain J. Math 27 (1997), 1041-1053. 


\bibitem{CH} {\sc  R. Courant, D. Hilbert}, {\it Methods of mathematical physics.} Vol. I. Interscience Publishers, Inc., New York, N.Y., 1953.

\bibitem{C} {\sc C. Cowan} Supercritical elliptic problems on a perturbation of the ball, {\em	Journal of Differential Equations}, {\bf 256} (2014) 1250-1263

\bibitem{DGG}{\sc Dancer, E.N., Gladiali, F., Grossi, M.}, 
On the Hardy-Sobolev equation
(2017) {\em Proceedings of the Royal Society of Edinburgh Section A: Mathematics}, 147 (2), pp. 299-336. 
DOI: 10.1017/S0308210516000135

\bibitem{DFetal} {\sc D. G. de Figueiredo, E. Moreira dos Santos, O. Hiroshi Miyagaki},
Sobolev spaces of symmetric functions and applications {\em Journal of Functional Analysis}, {\bf 261} (12), pp. 3735-3770 (2011). DOI: 10.1016/j.jfa.2011.08.016

\bibitem{DIPa}  {\sc F. De Marchis, I. Ianni, F. Pacella},
A Morse index formula for radial solutions of Lane-Emden problems
(2017) {\em Advances in Mathematics}, 322, pp. 682-737. DOI: 10.1016/j.aim.2017.10.026 

\bibitem{DIPb}  {\sc F. De Marchis, I. Ianni, F. Pacella},
	Exact Morse index computation for nodal radial solutions of Lane-Emden problems
	(2017) {\em Mathematische Annalen}, 367 (1-2), pp. 185-227. DOI: 10.1007/s00208-016-1381-6





\bibitem{GL} {\sc N. Garofalo, F.H. Lin} {Monotonicity properties of variational integrals, $A_p$ weights and unique continuation} (1986) {\it Indiana University Mathematics Journal}, {2}, pp. {245-268}. DOI: 10.2307/24893906


\bibitem{GT}{\sc D. Gilbarg, N. Trudinger}, {\it Elliptic Partial Differential Equations of Second Order.} Reprint of the 1998 edition. Classics in Mathematics. Springer-Verlag, Berlin, 2001
  
 \bibitem{GGN}  {\sc F. Gladiali, M. Grossi, S.L.N. Neves}, Nonradial solutions for the H\'enon equation in $\R^N$, {\em Advances in Mathematics}, {\bf 249} (2013), 1-36, doi:10.1016/j.aim.2013.07.022.

\bibitem{GGN2}{\sc F. Gladiali, M. Grossi, and S. L. N. Neves}, Symmetry breaking and Morse index of solutions of nonlinear elliptic problems in the plane {\em Commun. Contemp. Math.} {\bf 18} (2016), doi: 10.1142/S021919971550087X

\bibitem{GI}{\sc F. Gladiali, I. Ianni},
Quasi-radial nodal solutions for the Lane-Emden problem in the ball, (2017) arXiv:1709.03315

\bibitem{GGPS}{\sc F. Gladiali, M. Grossi,F. Pacella, P.N. Srikanth} Bifurcation and symmetry breaking for a class of semilinear elliptic equations in an annulus. {\em Calc. Var. Partial Differential Equations} 40 (2011), 295-317.

\bibitem{GGP} {\sc M.~Grossi, C.~Grumiau, F.~Pacella},  Lane Emden problems with large exponents and singular Liouville equations, {\em Journal des Mathematiques Pures et Appliquees} {\bf 101/6} (2014),  735-754, DOI: 10.1016/j.matpur.2013.06.011



\bibitem{KW} {\sc J. K\"ubler, T. Weth} Spectral asymptotics of radial solutions and nonradial bifurcation for the H\'enon equation, (2019) arXiv:1901.00453


\bibitem{LWZ} {\sc Z. Lou, T. Weth, Z. Zhang}, Symmetry breaking via Morse index for equations and systems of H\'enon-Schrodinger type, {\em Z. Z. Angew. Math. Phys.} (2019) 70: 35. https://doi.org/10.1007/s00033-019-1080-8

\bibitem{Ni} {\sc W.N. Ni,} A nonlinear Dirichlet problem on the unit ball and its applications.{\em Indiana Univ. Math. J.} 31 (1982), no. 6, 801–807. 


\bibitem{palais}{\sc R.S. Palais},  The Principle o f Symmetric Criticality, Commun. Math. Phys  69 (1979), 19-30

\bibitem{Struwe} {\sc M. Struwe},  Superlinear elliptic boundary value problems with rotational symmetry, {\em  Arc. Mat.} 39 (1982),  233-240.

\bibitem{talenti} {\sc G. Talenti},
Best constant in Sobolev inequality
{\em Annali di Matematica Pura ed Applicata}, Series 4, {110 (1)}, (1976), 353-372, DOI: 10.1007/BF02418013

\bibitem{Walter} {\sc W. Walter},
{\it Ordinary differential equations. } Graduate Texts in Mathematics, 182. Readings in Mathematics. Springer-Verlag, New York (1998).


\end{thebibliography}
\end{document}